\documentclass{article}

\usepackage{arxiv}

\usepackage[utf8]{inputenc} 
\usepackage[T1]{fontenc}    
\usepackage{hyperref}       
\usepackage{url}            
\usepackage{booktabs}       
\usepackage{amsfonts}       
\usepackage[english]{babel}
\usepackage{amsfonts}
\usepackage{amssymb}
\usepackage{xcolor}
\usepackage{amsthm,amsmath}
\usepackage{enumerate}
\usepackage{graphicx}
\usepackage{hyperref}
\usepackage{float}
\usepackage{hyperref}
\usepackage{array}
\usepackage{stmaryrd}
\usepackage{bbm}
\usepackage{subfig}
\usepackage[affil-it]{authblk}
\usepackage{xr}

\newtheorem{theorem}{Theorem}[section]
\newtheorem{lemma}[theorem]{Lemma}
\theoremstyle{definition}
\newtheorem{definition}{Definition}
\theoremstyle{remark}
\newtheorem*{remark}{Remark}
\theoremstyle{definition}

\newtheorem{proposition}[theorem]{Proposition}
\newtheorem{corollary}[theorem]{Corollary}

\title{Fast nonasymptotic testing and support recovery for large sparse Toeplitz covariance matrices}

\author[1, *]{Nayel Bettache}
 
\author[1]{Cristina Butucea}
    
\author[1]{Marianne Sorba }
    
\affil[1]{CREST, ENSAE Paris, 5 Avenue Le Chatelier, 91120 Palaiseau, FRANCE  \authorcr
  \{\tt nayel.bettache, cristina.butucea\}@ensae.fr \\
    marianne.sorba@gmail.com}
\affil[*]{Corresponding author}

\begin{document}

\maketitle

\begin{abstract}
		We consider $n$ independent $p$-dimensional Gaussian vectors with covariance matrix having Toeplitz structure. We test that these vectors have independent components against a stationary distribution with sparse Toeplitz covariance matrix, and also select the support of non-zero entries. We assume that the non-zero values can occur in the recent past (time-lag less than $p/2$). We build test procedures that combine a sum and a scan-type procedures, but are computationally fast, and show their non-asymptotic behaviour in both one-sided (only positive correlations) and two-sided alternatives, respectively. We also exhibit a selector of significant lags and bound the Hamming-loss risk of the estimated support. These results can be extended to the case of nearly Toeplitz covariance structure and to sub-Gaussian vectors. Numerical results illustrate the excellent behaviour of both test  procedures and support selectors - larger the dimension $p$, faster are the rates.
\end{abstract}


	\section{Introduction}
Covariance matrices of high-dimensional vectors appear in machine learning, signal processing and statistical procedures. In these fields, e.g. in the test-phase of an algorithm or in the validation step of a statistical model, the quality of the residuals (the difference between the observed and the predicted values) is a good indicator of the good performance of the procedure. More precisely, the closer the residuals are to a white noise distribution, the less information was lost by the predictor or the model at hand. It is therefore natural to look for very weak, sparse information in the covariance matrix of such residuals.

Goodness-of-fit tests are designed to assess whether the underlying (unknown) covariance matrix of high-dimensional vectors is the identity (which defines the null hypothesis), or it is far from it with respect to some distance (the alternative hypothesis). The separation radius is a measure of how far the covariance matrix needs to be from the identity matrix in order to be able to distinguish it given the observations.
Another important information is to recover the support of the covariance matrix, i.e. the set where the non-null values can be found. As in high-dimensional regression, this support is used to reduce dimension of the problem, produce unbiased estimators of the non-null entries and so on. A selector is a vector with coordinates taking value 1 when the covariance value is non-null, respectively 0 when it is null. We appreciate the quality of a selector in Hamming loss, which counts the number of miss-classified coordinates.
Our main interests are both testing the covariance matrix and recovering the support of significant covariance elements under the alternative hypothesis of weak sparse covariance values.

We consider the p-dimensional observations $X_1,...,X_n$ independent, with Gaussian probability distribution $\mathcal{N}_p(0,\Sigma)$, where $\Sigma = [\sigma_{ij}]_{1 \leq i,j, p}$ belongs to the set $\mathcal{S}_p^{++}$ of positive definite symmetric matrices. Let us denote by $X$ a generic vector with the same Gaussian $\mathcal{N}_p(0,\Sigma)$ distribution.
	
More particularly, when the vector $X$ is issued from a stationary process, its covariance matrix $\Sigma$ has a Toeplitz structure, that is its diagonal elements are all constant and we denote by $\sigma_{i,j} =\text{Cov}(X^i,X^j)= \sigma_{|i-j|}$ for all $i,\, j$ from 1 to $p$. As mentionned in \cite{Chenetal}, stationary time series are used as approximations of geometrically ergodic time series (whose transition probabilities converge exponentially fast to the stationary distribution). The information on the Toeplitz matrix is fully contained in the vector $(\sigma_0,\sigma_1,..., \sigma_{p-1} )$ of its diagonal values. More generally, we may study similarly any covariance matrix by looking at the energy of each diagonal of the covariance matrix, that is its euclidean norm $\sigma_k = \|(\sigma_{1,k+1}, ...,\sigma_{p-k,p})\|_2$.
Here, we devote our efforts to quantifying the benefits of the Toeplitz structure in terms of rates for testing and for support recovery.

\bigskip

{\bf Contributions} In this paper, we first give a new variant of concentration inequality for quadratic forms of large Gaussian vectors and specify these bounds for covariance matrices that are Toeplitz with few non-null diagonals.  We show non-asymptotic separation rates for testing large sparse Toeplitz covariance matrices which are remarkably fast due to the structure of the matrix. We test here whether the covariance matrix is the identity matrix $I_p$ or there exists a number $s$ of covariance elements among $\sigma_1,...,\sigma_{p-1}$ that are significantly positive (one-sided alternative), respectively significantly different from zero (two-sided alternative). The test procedure combines a sum and a scan procedure in order to detect small (relatively) numerous non-null entries and very few but sufficiently large entries, respectively. This is analogous to but more general than the detection of sparse Gaussian means (\cite{IngsterI,IngsterII}, \cite{DonohoJin}) where observations have the same variance, whereas our model is heteroscedastic.

Moreover, we propose a selector of the diagonals with non-null entries - a lag selector, which is constructed by universal thresholding of some linear estimators. We provide fast non asymptotic bounds for the expected value of its loss.

Experimental results show the excellent behaviour of these procedures with small values of $n$ (non-asymptotic character of our results) and large values of $p$. Indeed, by exploiting the Toeplitz structure, the matrix size $p$ does not act as a nuisance parameter anymore, but diminishes the convergence rates. All test procedures and the lag-selector are computationally trivial to implement. Note that the scan procedure is performed on a vector as well and it is therefore computationally fast, in contrast with the scan procedure of matrices, see e.g. \cite{ButuceaIngster} or \cite{castro_bubeck_lugosi_bernoulli}.

\bigskip

{\bf Bibliography} Previously, Cai and Ma \cite{cai2013} considered the same goodness-of-fit test with alternative characterized by covariance values that belong to an $\mathbb{L}_2$ ball of fixed radius. Tests for sparse covariance matrices were given by Arias-Castro, Bubeck and Lugosi \cite{castro_bubeck_lugosi_aos} and \cite{castro_bubeck_lugosi_bernoulli}. They considered alternative covariance matrices having at most $s$ significant values and also the structured alternative of a clique of size $s$ producing a small submatrix of significant values. Our testing rates are faster, but they are difficult to compare as the Toeplitz structure does not allow for the block or the clique sparsity structure in their paper.
Butucea and Zgheib \cite{butucea_zgheib_jmva} and \cite{butucea_zgheib_ejs} considered the test problem with alternatives that generalize the $\mathbb{L}_2$-ball in \cite{cai2013} to dense ellipsoids for both Toeplitz and not necessarily Toeplitz covariance matrices, respectively. More precisely, it was assumed that $\sigma_k$ decreased slowly as a polynomial (Sobolev ellipsoids) or faster, as an exponential of $k$. The test procedure involved an optimal banding parameter - specific for testing and different from the optimal parameter for estimation of the matrix. It was thus noticed that the minimax rates for goodness-of-fit testing of large covariance matrices are faster for Toeplitz matrices than for non Toeplitz ones, and that they are faster for testing than for estimation of the covariance matrix. In this paper, we consider an alternative class where at most $s$ significant values appear sparsely.

Cai and Liu \cite{doi:10.1198/jasa.2011.tm10560} and Cai, Liu and Xia \cite{cai_liu_xia_JASA} considered the problem of support recovery in the sense that the estimated set $\hat{\mathbf{S}}_n$ is different from the true set $\mathbf{S}$ with probability tending to 0. To the best of our knowledge, no quantitative rates were given for support recovery in the covariance matrix setup. In the context of Toeplitz covariance matrices, we call this problem lag-selection.

Our bounds for testing and lag selection are non-asymptotic, thus $n$ can be equal to 1 when we cannot observe repeated measurements. However, an important remark is that the rates are faster when the significant covariance values have lags in the recent past: $k \leq S$, for some $S < p$. Indeed, the rates depend on $p-S$.
From an asymptotic point of view, $s$ can tend to infinity as $p$ tends to infinity, thus we allow a nonparametric model (in the sense that the number of parameters increases). Such models have only been considered in nonparametric estimation of the spectral density of stationary time series, see Kreiss, Paparoditis and Politis \cite{paparo_politis} who uses thresholded empirical covariance coefficients.
	

	\section{Linear functionals of the covariance matrix}
	We define $\varphi_A$ the linear functional of the covariance matrix $\Sigma$ associated to the matrix $A$ belonging to $\mathcal{S}_p$ (the set of symmetric $p\times p$ matrices) as $\varphi_A(\Sigma) = \mathrm{Tr}(A\Sigma)$.
	%
	Recall that $\mathrm{Tr}(A^2)$ is also denoted by $||A||_F^2$, the Frobenius norm squared, for any $A$ in $\mathcal{S}_p$. We denote by $\|A\|_\infty$ the largest eigenvalue of the matrix $A$.
We recall that a centered real-valued random variable $Z$ is sub-exponential with positive parameters $(\nu^2, b)$ if \begin{equation}
		\mathbb{E}[\exp(t Z)]\leq \exp\left(\frac{\nu^2 t^2}{2}\right) \space \text{, for all } |t| \leq \frac 1b  . \label{subexp}
		\end{equation}

The sample covariance matrix is denoted
	\[
\Sigma_n=\frac{1}{n}\sum_{k=1}^nX_kX_k^T.
\]

The next theorem states that for $X_1,...,X_n$ independent multivariate Gaussian $\mathcal{N}_p(0,\Sigma)$ vectors, the random variable $Z = \varphi_A(\Sigma_n-\Sigma)$, for $A$ in $\mathcal{S}_p$, is sub-exponential with explicit values for the parameters $(\nu^2, b)$. We recall the Bernstein inequality that holds for sub-exponential random variables \cite{wainwright2019high}.

\begin{proposition} \label{Bernstein}
	If $Z$ is a sub-exponential random variables with parameters $(\nu^2, b)$, then
\begin{equation*}
	\mathbb{P}[Z \geq t]\leq \left\{\begin{array}{c}
	\exp\left(-\frac{t^2}{2\nu^2}\right) \text{ if }\hspace{0.2cm} 0\leq t \leq\frac{\nu^2}{b}\\
	\exp\left(-\frac{t}{2b}\right) \text{ if}\hspace{0.2cm} t >\frac{\nu^2}{b}\hspace{1cm}\\
	\end{array}\right.
\end{equation*}
	Equivalently, $Z$ verifies :
$$
\mathbb{P}[Z \geq t_u]\leq \exp\left(-\frac{u}{2}\right), \mbox{ for all } u>0,
$$
	where $t_u= \max(\nu\sqrt{u},bu)$.
\end{proposition}

Thus, the concentration inequality for the plug-in estimator $	\varphi_A(\Sigma_n)$ of $\varphi_A(\Sigma)$ follows immediately.

	\begin{theorem}\label{Theorem_subexp}
The random variable $\varphi_A(\Sigma_n-\Sigma)$ (respectively $\varphi_A(\Sigma-\Sigma_n)$) is centered and sub-exponential with parameters $\left(\nu^2=\frac{2||A\Sigma||_F^2}{n(1-K)},b=\frac{2||A\Sigma||_\infty}{nK}\right)$, for some arbitrary $K$ in $]0,1[$. Therefore, we have :
	\begin{equation}
	\mathbb{P}[\varphi_A\left(\Sigma_n-\Sigma\right)\geq t_u]\leq \exp\left(-\frac{u}{4}\right) \text{, for all } u>0, \label{eq:BI}
	\end{equation}
	where $t_u=\max\left\{\sqrt{u}\frac{||A\Sigma||_F}{\sqrt{n(1-K)}},u\frac{||A\Sigma||_\infty}{nK}\right\}$.
\end{theorem}
Previous concentration inequalities were given for such functionals. The closest to our case is the chi-square type concentration inequality in Spokoiny and Zhilova \cite{Spokoiny2013} for standardized Gaussian vectors and generalized to sub-Gaussian vectors. They generalized Hsu, Kakade and Zhang \cite{hsu2012} who assumed finite exponential moments of any order for the vector $X$. Let us also mention Giurcanu and Spokoiny \cite{giurcanu_spokoiny} who gave a Bernstein inequality for the empirical covariance element of a stationary centered Gaussian process and generalized it to locally stationary Gaussian processes.
	
Let us also mention the Hanson-Wright inequality which is stated for more general sub-Gaussian vectors but having independent components i.e. a diagonal covariance matrix (see Rudelson and Vershynin \cite{Rudelson_Ver} and its improvement under Bernstein condition on moments by Bellec \cite{Bellec2016ConcentrationOQ}).

The concentration inequality (\ref{eq:BI}) is the main tool in the applications that we consider hereafter to study stationary time series. In this context, we assume that $X_1,..,X_n$ are repeated, independent observations of length $p$ of an underlying stationary process $X = \{X^1,...,X^p,...\}$.  Note that our results are non-asymptotic, thus $n$ can be equal to 1. Without loss of generality, we assume that the process is centered. The covariance matrix of a stationary process is a Toeplitz covariance matrix, and we denote by $\sigma_{|i-j|}=\text{Cov}(X^i,X^j)$.
Let us denote by $\mathcal{T}_p$ the set of $p\times p$ Toeplitz matrices and by $|S|$ the cardinal of a set $S$.
\begin{definition}
	We define $\mathcal{F}_+(s,S,\sigma)$, for $\sigma >0$ real number and $s\leq S$ integer numbers between 1 and $p-1$, the set of sparse Toeplitz covariance matrices $\Sigma$ such that there are $s$ significantly positive  covariance elements with lags no larger than $S$ :
	\begin{eqnarray*}
\mathcal{F}_+(s,S,\sigma) &=& \left\{\Sigma\in \mathcal{S}_p^{++}\cap\mathcal{T}_p \mbox{ and }
\mbox{ there exists } \mathbf{S} \subseteq \{1,...,S\} \text{  such that  } |\mathbf{S}|=s, \right.\\
&& \left. \sigma_j \mbox{ is } \left\{\begin{array}{ll}
	\geq\sigma> 0, &\text{ for all }j\in \mathbf{S}\\
	=0, & \text{ for all }j\in \{1,p-1\} \backslash \mathbf{S}\\
	\end{array}\right\}\right\}.
\end{eqnarray*}
Similarly, we define the two-sided set $\mathcal{F}(s,S,\sigma)$ :
	\begin{eqnarray*}
\mathcal{F}(s,S,\sigma) &=& \left\{\Sigma\in \mathcal{S}_p^{++}\cap\mathcal{T}_p \mbox{ and }
\mbox{ there exists } \mathbf{S} \subseteq \{1,...,S\} \text{  such that  } |\mathbf{S}|=s, \right.\\
&& \left. |\sigma_j |\mbox{ is } \left\{\begin{array}{ll}
	\geq\sigma> 0, &\text{ for all }j\in \mathbf{S}\\
	=0, & \text{ for all }j\in \{1,p-1\} \backslash \mathbf{S}\\
	\end{array}\right\}\right\}.
\end{eqnarray*}
\vspace{0.8cm}
\end{definition}

Let us apply Theorem \ref{Theorem_subexp} to several choices of the matrices $A$. First, the covariance element $\sigma_j$, $j \geq 1$, can be written as $\sigma_j=\mathbb{E}[X^TA_jX]=\mathrm{Tr}(A_j\Sigma)$, with $[A_j]_{k\ell} = \frac 1{2(p-j)} I(|k-\ell| = j)$ - a matrix that has 0 elements except on $j$th upper and lower diagonals. Note that we use notation $A_j$ instead of $A_{\{j\}}$. The empirical estimator of $\sigma_j$ can be written as
\[
\hat{\sigma_{j}} 
=\frac{1}{n}\sum_{k=1}^{n}X_k^T A_j X_k=\mathrm{Tr}(A_j \Sigma_n).
\]

	\begin{remark}{\rm
		It is useful to note that our results can be generalized to time series that are "nearly" stationary, by considering :
\[
\tilde{\sigma_j}=\mathrm{Tr}(A_j \Sigma_n) =\frac{1}{2(p-j)}\sum_{i,k=1,|i-k|=j}^{p}\sigma_{i,k} .
\]
In that case, we consider slightly different sets of sparse covariance matrices: $\widetilde{\mathcal{F}}_+(s,S,\sigma)$ and $\widetilde{\mathcal{F}}(s,S,\sigma)$ of not necessarily Toeplitz matrices with $s$ diagonal average values $\tilde{\sigma_j}$ of the first $S$ being significant. By taking into consideration that all methods that we study in the sequel for testing and lag selection are exclusively based on the concentration of the mean empirical correlations around their expected values $\tilde{\sigma_j}$, the following results remain valid provided that we control $||A \Sigma||_F$ and $||A \Sigma||_\infty$.}
	\end{remark}
 Let $W \subseteq \{1,...,S\}$ be a set of $w$ values between 1 and S. We estimate
 \[
 \sum_{j\in W}\sigma_j = \mathrm{Tr}(A_W \Sigma), \quad \text{where }A_W=\sum_{j\in W}A_j
 \] by $\mathrm{Tr}(A_W\Sigma_n)$.

	\begin{proposition}\label{prop_norm}
	Let $W \subseteq \{1,...,S\}$ contain $w$ elements and  $A_W=\sum_{j\in W}A_j$. We have :
	\begin{enumerate}
		\item $||A_W||_\infty\leq\frac{w}{p-S}$ and $||A_W||_F^2 \leq\frac{w}{2(p-S)}$
		\item For any covariance matrix $\Sigma$ belonging to   $\mathcal{F}(s,S,\sigma)$, \\ $||A_W\Sigma||_\infty \leq \sigma_0 \frac{w(2s+1)}{p-S}$ and
$	||A_W \Sigma||_F^2\leq \sigma_0^2 \cdot \left\{\begin{array}{c}
		\frac{\mathcal{K}(2s+1)}{(p-S)}, \quad \text{if}\quad  w=1\\
		\frac{w(2s+1)^2}{2(p-S)}, \quad \text{if}\quad w>1
		\end{array}\right.$ \\
		where $\mathcal{K}=\left\{\begin{array}{c}
		1,\quad \text{if} \quad W \subseteq\{1,...,\frac{p}{2}-1\}\\
		\frac{p}{2},\quad \text{if} \quad W \subseteq\{\frac{p}{2},...,p-1\}.\hspace{0.5cm}			 \end{array}\right.$
	\end{enumerate}
\end{proposition}


The next Corollary specifies the concentration inequality in Theorem~\ref{Theorem_subexp} using the bounds in the Proposition \ref{prop_norm} above.

\begin{corollary}\label{corollary_1}
Let $X_1,..,X_n$ be i.i.d, $\mathcal{N}_p(0_p,\Sigma)$, $\Sigma $ belonging to
$ \mathcal{F}_+(s,S,\sigma)$ or $\mathcal{F}(s,S,\sigma)$ and $W \subseteq \{1,...,S\}$ with $S<\frac{p}{2}$ having $w$ elements. Then, for some arbitrary $K$ in $]0,1[$,
\begin{equation}
\mathbb{P}_{{I_p}}[\varphi_{A_W}\left(\Sigma_n-{I_p}\right)\geq \sigma_0 \cdot t ]\leq \exp\left(-\frac{u}{4}\right),\quad\text{for all}\quad u> 0,\label{erreur1}
\end{equation}  	where $$t=\max\left\{\sqrt{\frac{u}{2(1-K)}}\sqrt{\frac{w}{n(p-S)}},\frac{u}{K}\frac{w}{n(p-S)}\right\}.$$
Moreover, for any $\Sigma$ in $\mathcal{F}(s,S,\sigma)$,
\begin{equation}
\mathbb{P}_{\mathrm{\Sigma}}[\varphi_{A_W}\left(\Sigma_n-\Sigma\right)\geq \sigma_0 \cdot\tilde{t}]\leq \exp\left(-\frac{u}{4}\right),\quad\text{for all}\quad u>0,\label{erreur2}\end{equation} 	where $$\tilde{t}=\max\left\{\sqrt{\frac{u}{(1-K)}}\sqrt{\frac{2s+1}{n(p-S)}},\frac{u}{K}\frac{2s+1}{n(p-S)}\right\}\quad \text{if}\quad w=1$$ and $\tilde{t}=(2s+1)t$ if $w>1$.
\end{corollary}
\vspace{0.8cm}

Similar inequalities hold for $|\varphi_{A_W}\left(\Sigma_n-{I_p}\right)|$ and $ |\varphi_{A_W}\left(\Sigma_n-\Sigma\right)|$ multiplies the exponential term by a factor two in (\ref{erreur1}) and (\ref{erreur2}).

If $W=\{1,...,S\}$, it is enough to replace $w$ by $S$ in the previous results. However, if $W=\{j\}$ for some $j\leq S$, the previous results are still true with $w$ replaced by 1.

From now on, we assume that $S< \frac{p}{2}$ such that $\mathcal{K}=1$ in the previous Proposition. Indeed, in the context of time series, it is natural to look for significant correlations in the recent past.


\section{Non-parametric testing for stationary time series}
From now on, we assume for simplicity that $\sigma_0=1$, thus dealing with correlation matrices only.
The one-sided test problem is
\[
H_0 : \Sigma = {I_p}, \quad \text{vs. } H_1 : \Sigma \in \mathcal{F}_+(s,S,\sigma).
\] \\ The following two-sided test problem will also be discussed as a generalization
\[
H_0 : \Sigma = {I_p} \quad \text{vs. } H_1 : \Sigma \in \mathcal{F}(s,S,\sigma).
\] \\
Recall that a test procedure $\Delta_n$ is a binary valued random variable $\Delta_n : (\mathbb{R}^{p})^{\otimes n}\rightarrow\{0,1\}$. It separates the set of possible outcomes of some random event in two contiguous sets, we decide to reject $H_0$ whenever $\Delta_n=1$ and to accept $H_0$ whenever $\Delta_n=0$. The maximal testing risk is defined as
\[
R(\Delta_n,\mathcal{F}_+)=\mathbb{P}_{I_p}(\Delta_n=1)+\sup_{\Sigma \in \mathcal{F}_+}\mathbb{P}_\Sigma(\Delta_n=0),
\]
that is the sum of the type $I$ and the maximal type $II$ error probabilities over the set in the alternative hypothesis. A separation rate is the least possible value for $\sigma>0$ such that the maximal testing risk stays below some prescribed value.

We proceed by considering successively two measures of the separation between ${I_p}$ and $\Sigma$ under the alternative hypothesis $H_1$. We choose successively the sets $W = \{1,...,S\}$ and $W = \mathbf{S}$, and arbitrary subset of $\{1,...,S\}$ with $s$ elements.
For testing over $\mathcal{F}_+(s,S,\sigma)$, we consider
\[
\mathrm{Tr}(A_{1:S})
\quad \text{and }
\max_{\mathbf{S} \subseteq \{1,...,S\}, \#\mathbf{S}=s}\mathrm{Tr}(\mathcal{A}_\mathbf{S}\Sigma).
\]
Correspondingly, over $\mathcal{F}(s,S,\sigma)$ we consider
\[\sum_{j=1}^{S}|\sigma_j|=\sum_{j=1}^{S}|\mathrm{Tr}(A_j\Sigma)|
 \quad \text{and }
  \max_{\mathbf{S} \subseteq \{1,...,S\}, \#\mathbf{S}=s}\sum_{j\in \mathbf{S}}^{S}|\mathrm{Tr}(A_j\Sigma)|.
  \]

By analogy to the vector case, we distinguish moderately sparse and highly sparse covariance structure. In the first case, the sum of all $S$ values will allow to test, whereas in the latter a search over subsets of size $s$ will be necessary. This is called a scan procedure and it is computationally fast for vectors. Note that, if the sparsity $s$ is unknown a second search over different possible values of $s$ will produce an aggregated procedure, free of $s$.

\subsection{Moderately sparse covariance structure}
When the alternative hypothesis is $\mathcal{F}^+(s,S,\sigma)$, we consider for some $t_{n,p}^{MS+}$ the test procedure
\begin{equation}
\Delta_n^{MS+}=I\left(\varphi_{A_{1:S}}(\Sigma_n - I_p)\geq t_{n,p}^{MS+}\right). \label{test_MS}
\end{equation}
\begin{theorem}\label{Theorem_testMS}
	The test $\Delta_n^{MS+}$ defined in (\ref{test_MS}), with
\[
t_{n,p}^{MS+}=\max\left\{\sqrt{\frac{u \cdot S}{n(p-S)}}, \frac{2u \cdot S}{n(p-S)}\right\}
\] for $u>0$ is such that
\[
R(\Delta_n^{MS+},\mathcal{F}^+)\leq 2\exp\left(-\frac{u}{4}\right)
\] provided that $\sigma \geq \frac{2(s+1)}{s}t_{n,p}^{MS+}$.
\end{theorem}
When the alternative hypothesis is $\mathcal{F}(s,S,\sigma)$, we consider for some $t_{n,p}^{MS}$ the test procedure
\begin{equation}
\Delta_n^{MS}=I\left(\sum_{i=1}^{S}|\varphi_{A_i}(\Sigma_n - I_p)|\geq t_{n,p}^{MS}\right). \label{test_MSG}
\end{equation}
\begin{theorem}\label{Theorem_testMS2}
	The test $\Delta_n^{MS}$ defined in (\ref{test_MSG}), with
\[
t_{n,p}^{MS}=S\max\left\{ \sqrt{\frac{4u \log(S)}{n(p-S)}},\frac{8 u \log(S)}{n(p-S)}
\right\}
\]for $u>1$ is such that $$R(\Delta_n^{MS},\mathcal{F})\leq 4\exp\left(-(u-1)\log (S)\right)$$
	 provided that $\sigma \geq t_{n,p}^{MS}+\max\left\{ \sqrt{\frac{4 (u-1) (2s+1) \log (S)}{n(p-S)}}, \frac{8 (u-1)(2s+1) \log (S)}{n(p-S)}\right\}$ .
\end{theorem}

\subsection{Highly sparse covariance structure}
Let us consider now for some threshold $t_{n,p}^{HS+}$ the test procedure
\begin{equation}\label{test_HS}
\Delta_n^{HS+}=\max_{\mathbf{S} \subseteq \{1,...,S\}, \#\mathbf{S}=s}I\left(\varphi_{\mathcal{A}_\mathbf{S}}(\Sigma_n- I_p)\geq t_{n,p}^{HS+}\right).
\end{equation}
The test $\Delta_n^{HS+}$ successively tries all possible sets $\mathbf{S}$ of $s$ diagonals among the first $S$ diagonal values. If any of these tests decides to reject $H_0$, then $\Delta_n^{HS+}$ also rejects $H_0$, otherwise $\Delta_n^{HS+}$ accepts the null hypothesis $H_0$.
\begin{theorem}\label{theorem_testHS}
	The test $\Delta_n^{HS+}$ defined in (\ref{test_HS}), with
\[
t_{n,p}^{HS+}=\max\left\{\sqrt{\frac{4u \cdot s \log{\binom{S}{s}}}{n(p-S)}},
\frac{8u \cdot s\log{\binom{S}{s}}}{n(p-S)}\right\}
\]
for $u>1$ is such that $$R(\Delta_n^{HS+},\mathcal{F}^+)\leq \exp\left(-(u-1)\log{\binom{S}{s}}\right)+\exp\left(-\frac{u}{4}\right)$$ provided that
$\sigma \geq \frac{1}{s}\left(t_{n,p}^{HS+}+(2s+1)\max\left\{\sqrt{\frac{u\cdot s}{n(p-S)}}, \frac{2u \cdot s}{n(p-S)}\right\}\right)$
\end{theorem}
When the alternative set of hypotheses is $\mathcal{F}(s,S,\sigma)$, consider  for some threshold $t_{n,p}^{HS}>0$
\begin{equation}\label{test_HSG}
\Delta_n^{HS}=\max_{\mathbf{S} \subseteq \{1,...,S\}, \#\mathbf{S}=s}I\left(\sum_{j \in \mathbf{S}}|\varphi_{A_j}(\Sigma_n- {I_p})|\geq t_{n,p}^{HS}\right).
\end{equation}
\begin{theorem}\label{theorem_testHSG}
	The test $\Delta_n^{HS}$ defined in (\ref{test_HSG}), with
\[
t_{n,p}^{HS}=s\max\left\{  \sqrt{\frac{4u \log\left(s{\binom{S}{s}}\right)}{n(p-S)}},
\frac{8u \log\left(s{\binom{S}{s}}\right)}{n(p-S)}\right\}
\]
for $u>1$ is such that
\[R(\Delta_n^{HS},\mathcal{F})\leq 4\exp\left[-(u-1)\log\left(s{\binom{S}{s}}\right)\right]
 \] provided that $\sigma \geq t_{n,p}^{HS} + \max\left\{\sqrt{\frac{4 (u-1) \log\left(s(2s+1){\binom{S}{s}}\right)}{n(p-S)}}, \frac{8 (u-1) \log\left(s(2s+1){\binom{S}{s}}\right)}{n(p-S)}\right\}$.
\end{theorem}
\begin{remark}
{\rm	When the separation is measured by $\max_{\mathbf{S}}\sum_{j\in \mathbf{S}}\sigma_j$, its estimator is known as the scan statistic. Note that the computations are not very involved. Indeed, after computing $\xi_1=\varphi_{A_1}(\Sigma_n- {I_p}),...,\xi_S=\varphi_{A_S}(\Sigma_n-{I_p})$, we sort these values in decreasing order : $\xi_{(1)}\geq\xi_{(2)}\geq...\geq \xi_{(S)}$, and then $$\max_{\mathbf{S} \subseteq \{1,...,S\}, \#\mathbf{S}=s}\sum_{j \in \mathbf{S}}\varphi_{A_j}(\Sigma_n- {I_p})=\xi_{(1)}+...+\xi_{(s)}$$
	Similar calculations hold for $\max_{\mathbf{S}}\sum_{j\in \mathbf{S}}|\sigma_j|$ and $|\xi|_{(1)}\geq|\xi|_{(2)}\geq...\geq |\xi|_{(S)}$. We thus exploit the Toeplitz structure that reduces the matrix structure to a vector and makes the scan statistic computationally efficient.}
\end{remark}
\begin{remark}
Note that the previous tests must be agregated over a set of possibel values for $s$ in order to be free of the sparsity $s$: $\tilde \Delta_n^{HS} = \max_s \Delta_n^{HS}$ will reject whever at least one test rejects.
\end{remark}

\vspace{-.5cm}

{\bf Discussion}
a) If $S \asymp \log(p) $, giving $p-S \asymp p$, the series has short memory. We get $t_{np}^{MS+}\asymp \sqrt{{\log(p)}/{(np)}}$ giving a test rate smaller than $\sqrt{{\log(p)}/{(np)}}$, and  with Stirling's approximation, $t_{np}^{HS+}\asymp s\sqrt{{\log\left(\frac{\log(p)}{s}\right)}/{(np)}}$ giving the following bound for the testing rate 
$
\sqrt{\frac{\log\left({\log(p)}/{s}\right)}{np}}+ \sqrt{\frac{s}{np}}$. 

We see that $\Delta_n^{HS+}$ detects smaller values of $\sigma$ than $\Delta_n^{MS+}$ when $s\leq \log(p)$, hence our choice to name the procedures $MS$ and $HS$ respectively.

b) If the stationary time series has longer memory, for example $S = {p}/{2}-1$, this gives  $p-S = {p}/{2}+1$ and $\frac{S}{p-S}\asymp 1$. In this case, $t_{np}^{MS+}\asymp {1}/{\sqrt{n}}$ and $\sigma \geq {1}/{\sqrt{n}}$, while $t_{np}^{HS+} \asymp s\sqrt{\frac{\log(p/s)}{np}}+ \sqrt{\frac{s}{np}}$.\\
Again, if $s/p \rightarrow 0$, the test $\Delta_{n}^{HS+}$ detects smaller values of $\sigma$ then $\Delta_{n}^{MS+}$. However, if $s=S\asymp \frac{p}{2}$, it is sufficient to use only $\Delta_{n}^{MS+}$.

Table~\ref{tablerates} summarizes our results where $C_1,\, C_2, \, C^*_1$ and $C^*_2$ denote constants depending only on $u$.

\vspace{-1cm}
\begin{table}[h!]
\caption{Thresholds $t$ and separation rates for moderately and highly sparse tests}
\begin{center} \label{tablerates}
 \begin{tabular}{|m{4mm}|m{7cm}|m{8cm}|}
 	\hline
 	\multicolumn{3}{|c|}{One sided test}\\
 	\hline
 	&$MS+$&$HS+$\\
 	\hline
 	 $t=$&$\max\left\{C_1\sqrt{\frac{S}{n(p-S)}}\},C_2\frac{S}{n(p-S)}\right\}$ 
 &$\max\left\{C_1\sqrt{\frac{s\log{\binom{S}{s}}}{n(p-S)}},C_2\frac{s\log{\binom{S}{s}}}{n(p-S)}\right\}$ \\
 	\hline
 	$\sigma\geq$&$\frac{2(s+1)}{s}t$&$\frac{t}{s}+ \frac{2s+1}s \max\left\{C_1\sqrt{\frac{s}{n(p-S)}},C_2\frac{s}{n(p-S)}\right\}$\\
 	\hline 
 	\end{tabular} \\
 	
 	\vspace{2em}
 	
 	\begin{tabular}{|m{4mm}|m{7cm}|m{8cm}|}
 	\hline
 	\multicolumn{3}{|c|}{Two sided test}\\
 	\hline
 	 &$MS$&$HS$\\
 	 \hline
 	  $t=$&$C\max\left\{C_1\sqrt{\frac{\log(S)}{n(p-S)}}\},C_2\frac{\log(S)}{n(p-S)}\right\}$
 &$s\max\left\{C_1\sqrt{\frac{\log\left(s{\binom{S}{s}}\right)}{n(p-S)}},C_2\frac{\log\left(s{\binom{S}{s}}\right)}{n(p-S)}\right\}$\\
 	  \hline
 	  $\sigma\geq$&$t+\max\left\{C^*_1\sqrt{\frac{(2s+1)\log(S)}{n(p-S)}},C^*_2\frac{(2s+1)\log(S)}{n(p-S)}\right\}$
 &$t+\max\left\{C^*_1\sqrt{\frac{\log\left(s(2s+1){\binom{S}{s}}\right)}{n(p-S)}},C^*_2\frac{\log\left(s(2s+1){\binom{S}{s}}\right)}{n(p-S)}\right\}$\\
 	  \hline
 \end{tabular}
 \end{center}
 \end{table}
 \vspace{-1cm}

{\bf Experimental results} A more detailed numerical study is included in the Section~5 Simulation results, including an example
of a sparse $MA(\lfloor p/4 \rfloor)$ series with increasing $p$.  We want to give a fast glimpse of the graphs of the power function, $\mathbb{E}_\Sigma(\Delta_n = 1)$, for the tests $\Delta_n^{MS}$ and $\Delta_n^{HS}$, for different values of $\Sigma$. Here $S=\sqrt{p}$ and $s= (S-1)/2$. Figures\footnote{All figures should be printed in color} ~\ref{fig:MS_main} and ~\ref{fig:HS_main}  show the power for different values of $p$ and $n$ as function of $\sum_{j=1}^S |\sigma_j|$ and $\sum_{j \in \mathbf{S}}|\sigma_j|$ - in a logarithmic scale that allow to better read this graphics. The plots show very steep power functions, that indicate a narrow band where the decision is hard to make. The power goes from small values near $\alpha=10\%$ to high values close to 1 in a fast increasing way. There are little differences in the behaviour of moderately and highly sparse tests.

We note an improvement as $p$ grows (the tests detect matrices closer to the identity), in agreement with theoretical rates that first indicated that $p$ is not a nuisance parameter here.
\begin{figure}[H]
\centering
\begin{minipage}{.32\textwidth}
\centering
\hspace*{\stretch{1}}%
\subfloat[$n=100$]{\includegraphics[width=4.5cm,height=4.5cm]{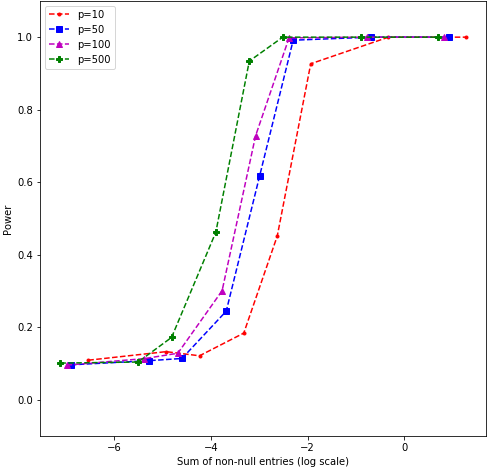}}%
\hspace{\stretch{1}}
\end{minipage}
\begin{minipage}{.32\textwidth}
\centering
\hspace*{\stretch{1}}%
\subfloat[$n=500$]{\includegraphics[width=4.5cm,height=4.5cm]{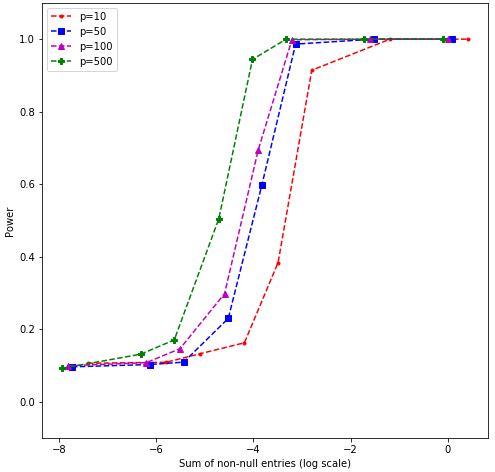}}%
\hspace{\stretch{1}}
\end{minipage}
\begin{minipage}{.32\textwidth}
\centering
\hspace*{\stretch{1}}%
\subfloat[$n=1000$]{\includegraphics[width=4.5cm,height=4.5cm]{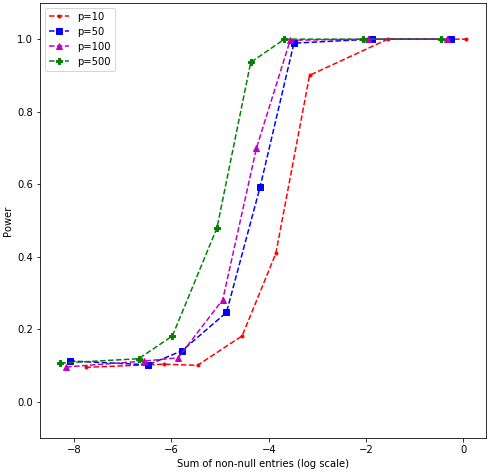}}%
\hspace{\stretch{1}}
\end{minipage}
\caption{Power of the $\Delta^{MS}_n$ test}
\label{fig:MS_main}
\end{figure}

\vspace{-1cm}

\begin{figure}[H]
\centering
\begin{minipage}{.32\textwidth}
\centering
\hspace*{\stretch{1}}%
\subfloat[$n=100$]{\includegraphics[width=4.5cm,height=4.5cm]{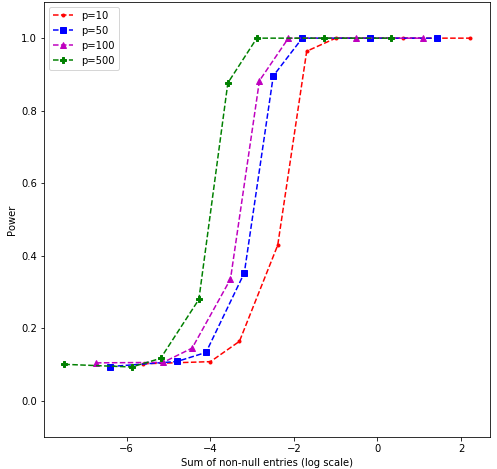}}%
\hspace{\stretch{1}}
\end{minipage}
\begin{minipage}{.32\textwidth}
\centering
\hspace*{\stretch{1}}%
\subfloat[$n=500$]{\includegraphics[width=4.5cm,height=4.5cm]{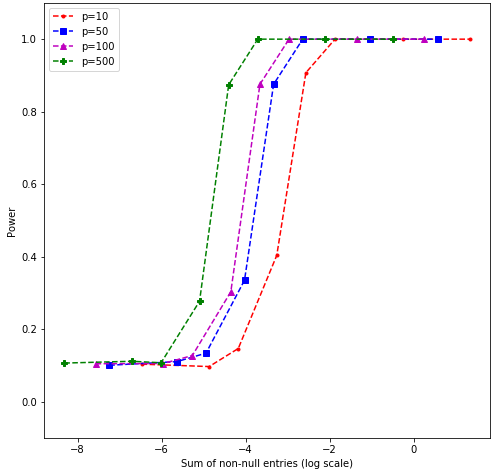}}%
\hspace{\stretch{1}}
\end{minipage}
\begin{minipage}{.32\textwidth}
\centering
\hspace*{\stretch{1}}%
\subfloat[$n=1000$]{\includegraphics[width=4.5cm,height=4.5cm]{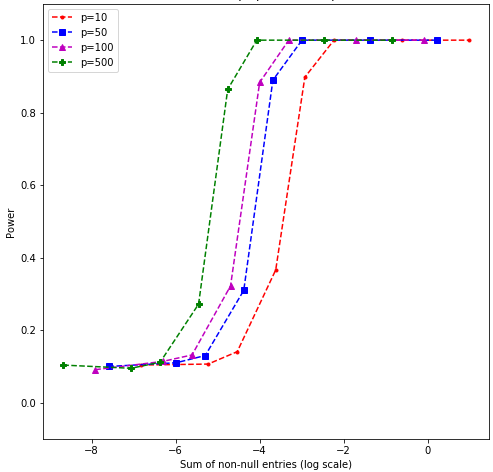}}%
\hspace{\stretch{1}}
\end{minipage}
\caption{Power of the $\Delta^{HS}_n$ test}
\label{fig:HS_main}
\end{figure}

\section{Lag-selection for stationary time-series}
 The objective here is to properly select non-null correlation coefficients. We define a (two-sided) lag-selection problem as estimation of $\eta$, a vector with entries $\eta_{j}=\mathbbm{1}(|\varphi_{A_j}\left(\Sigma\right)|>0)$. We want to find a selector $\hat \eta$ with $\hat{\eta_{j}}=\mathbbm{1}(|\varphi_{A_j}\left(\Sigma_n\right)|>\tau_n)$ that is consistent in the sense that the risk
\begin{equation*}
R^{LS}(\hat{\eta},\mathcal{F})=\sum_{j=1}^{S}\mathbb{E}_{\Sigma}[|\hat{\eta_{j}}-\eta_{j}|]
\end{equation*}
stays bounded (is small). The Hamming loss counts the number of miss-classified elements.
\begin{theorem}\label{theorem_select}
If $\Sigma$ belongs to $\mathcal{F}(s,S,\sigma)$, with $\sigma \geq 2 \tau_n$, the selector $\hat \eta$ with	 \[
\tau_n=\max\left\{\left(\sqrt{\log(s)}+\sqrt{\log(S-s)}\right)\sqrt{u\frac{2s+1}{n(p-S)}},
2u \log(s(S-s))\frac{2s+1}{n(p-S)}
\right\}
\] for $u>1$ is such that
	\begin{equation*}
	R_{LS}(\hat \eta,\mathcal{F})\leq 2\exp\left(-(u-1)\frac{\log(s)}{4}\right) + 2\exp\left(-(u-1)\frac{\log(S-s)}{4} \right).
	\end{equation*}
\end{theorem}
\begin{remark} {\rm
If we only consider the class $\mathcal{F}^+$, with  $\sigma>2 \tau_n$, we define a one-sided selection by $\eta^+_{j}=\mathbbm{1}(\varphi_{A_j}\left(\Sigma\right) > 0)$ and consider $\hat{\eta_{j}}^+ =\mathbbm{1}( \varphi_{A_j}\left(\Sigma_n\right) >\tau_n)$. Then
	\begin{equation*}
	R_{LS}(\hat \eta^+,\mathcal{F})\leq \exp\left(-(u-1)\frac{\log(s)}{4}\right) + \exp\left(-(u-1)\frac{\log(S-s)}{4} \right).
	\end{equation*}
}
\end{remark}
Take for example $S = \frac p2 - 1$, and assume that $s/p = p^{-\beta}$ for some $\beta$ in (0,1). This implies that $\log(S-s) \sim (1-\beta) \log(p)$ and the asymptotic value of $\tau_n$ as $p$ tends to infinity is
$$
\tau_n \sim (1 + \sqrt{1-\beta}) \sqrt{2 u \frac{\log(p)}{n p^{\beta}} }, \quad u>1 .
$$
Figure ~\ref{fig:LS} shows the good behaviour of our lag selector under $\Sigma \in \mathcal{F}(s, S, \sigma)$ hypothesis.
We plot the Hamming loss between $\eta$ and $\hat{\eta}$, averaged over 1000 repetitions, as a function of $n$, for numerous values of $p$ and taking $S = \sqrt{p}$. We note the fast decrease to 0 of the Hamming loss for both for $s=S-1$ and for $s=(S-1)/2$, despite the small values of $\sigma \asymp \tau_n $ to detect.

\vspace{-.8cm}

\begin{figure}[H]
\centering
\begin{minipage}{.49\textwidth}
\centering
\hspace*{\stretch{1}}%
\subfloat[$s = S-1$]{\includegraphics[width=4.5cm,height=4.5cm]{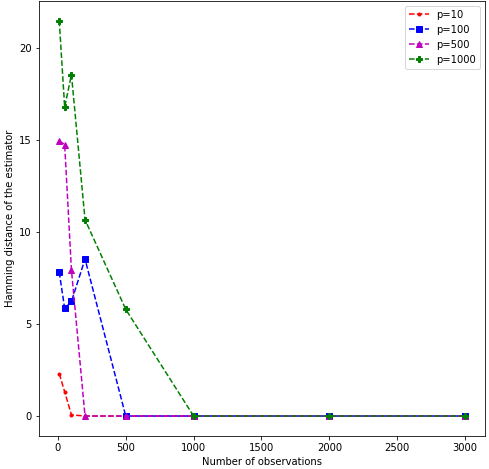}}%
\hspace{\stretch{1}}
\end{minipage}
\begin{minipage}{.49\textwidth}
\centering
\hspace*{\stretch{1}}%
\subfloat[$s = \frac{S-1}{2}$]{\includegraphics[width=4.5cm,height=4.5cm]{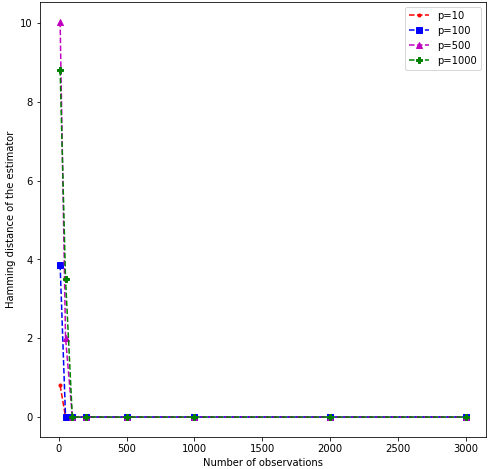}}%
\hspace{\stretch{1}}
\end{minipage}
\caption{Hamming-loss of the lag selector}
\label{fig:LS}
\end{figure}

\vspace{-1cm}

\section{Simulation results}

\subsection{Power curves of our test procedures}
We include several examples to illustrate the numerical behavior of our test procedures. First, we highlight that the plots will be drawn with a logarithmic scale. We estimate the power of the four test procedures: $\Delta_n^{MS+}$, $\Delta_n^{MS}$, $\Delta_n^{HS+}$, $\Delta_n^{HS}$ to test the null hypothesis $\Sigma = I$.

We choose the numbers of non-null entries $s$ and the non-null entries support $\textbf{S} \subset \llbracket 1; S \rrbracket$ with
\[ s = {(S-1)}/{2}, \quad \text{ and }
S=\sqrt{p}.
 \]
The location of the non zero entries is randomly chosen. We define the common value of non-null entries as growing fractions of $\sigma$. The threshold of the test procedure is defined as $t = t_{n, p, \alpha}$ the empirical $(1-\alpha)$-quantile of the test statistic under the null hypothesis. In order to determine its value empirically, we generate 5000 repeated samples under the null hypothesis. The plots represent the power of the tests by the measure of separation, namely:
\[\sum \limits_{j=1}^S \sigma_j, \text{ for the one sided tests, and }
\sum \limits_{j=1}^S |\sigma_j|, \text{ for the two-sided tests.}
 \]
 To generate the plots, we sample 5000 times under the alternative hypothesis and plot the mean value of the power of the tests. The $\alpha$ value will always be $0.1$.
\begin{figure}[H]
\centering
\begin{minipage}{.49\textwidth}
\centering
\hspace*{\stretch{1}}%
\subfloat[Logarithmic scale]{\includegraphics[width=4.5cm,height=4.5cm]{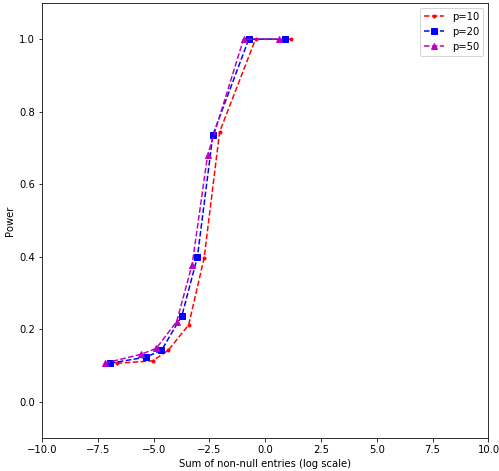}}%
\hspace{\stretch{1}}
\end{minipage}
\begin{minipage}{.49\textwidth}
\centering
\hspace*{\stretch{1}}%
\subfloat[Identity scale]{\includegraphics[width=4.5cm,height=4.5cm]{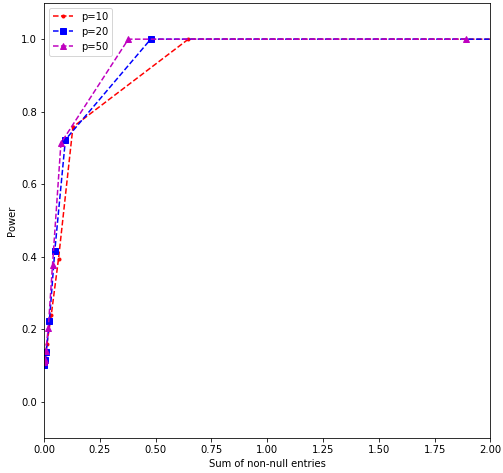}}%
\hspace{\stretch{1}}
\end{minipage}
\caption{Impact of the $x$-axis scale ($\Delta^{MS+}_n$ test)}
\label{fig:log_scale}
\end{figure}

Figure ~\ref{fig:log_scale} shows that the logarithmic scale should be preferred as it helps to better understand the behaviour of the test procedure when the measure of separation increases.

We represent now the power of the $\Delta^{MS+}_n$ test procedure as a function of the measure of separation for numerous values of $n$ and $p$. The best power function goes the fastest from low values above $\alpha = 0.1$ to high values close to 1. The change happens around the theoretical value of the  separation rate.

\begin{figure}[H]
\centering
\begin{minipage}{.32\textwidth}
\centering
\hspace*{\stretch{1}}%
\subfloat[$n=100$]{\includegraphics[width=4.5cm,height=4.5cm]{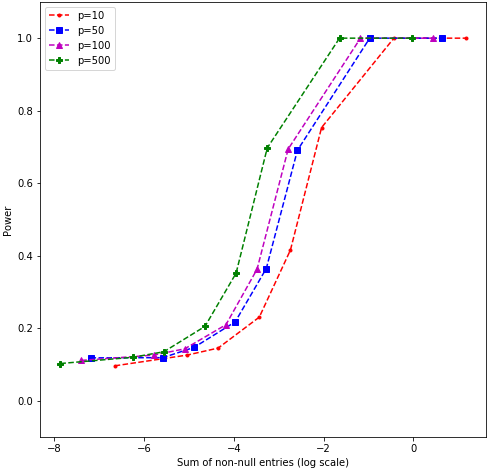}}%
\hspace{\stretch{1}}
\end{minipage}
\begin{minipage}{.32\textwidth}
\centering
\hspace*{\stretch{1}}%
\subfloat[$n=500$]{\includegraphics[width=4.5cm,height=4.5cm]{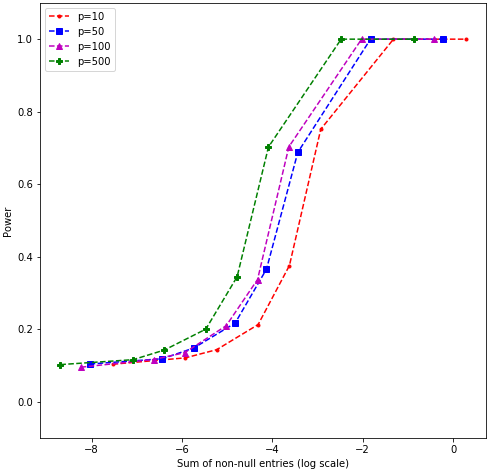}}%
\hspace{\stretch{1}}
\end{minipage}
\begin{minipage}{.32\textwidth}
\centering
\hspace*{\stretch{1}}%
\subfloat[$n=1000$]{\includegraphics[width=4.5cm,height=4.5cm]{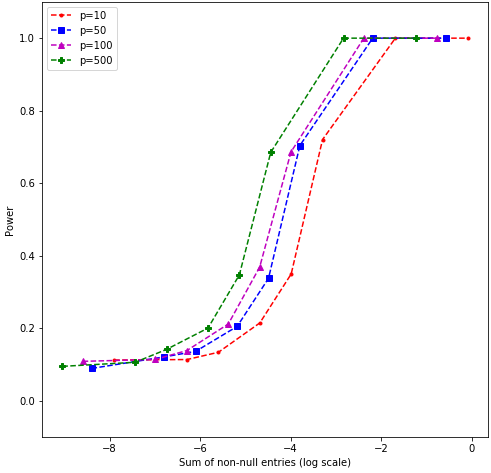}}%
\hspace{\stretch{1}}
\end{minipage}
\caption{Power of the $\Delta^{MS+}_n$ test}
\label{fig:MS+}
\end{figure}

Figure ~\ref{fig:MS+} shows that for $p$ smaller than, equal to or bigger than $n$, the $\Delta^{MS+}_n$ test presents similar behaviour as the measure of separation increases. However, it can be noticed that the performances are better in high dimension, that is the power curves are shifted to the left. This is in agreement with our theoretical rates and indicates that $p$ is not a nuisance parameter. The $\Delta^{MS+}_n$ test is not only robust but also more efficient in high dimension.

Let us consider the two-sided $\Delta^{MS}_n$ test and plot its estimated power curve.

\begin{figure}[H]
\centering
\begin{minipage}{.32\textwidth}
\centering
\hspace*{\stretch{1}}%
\subfloat[$n=100$]{\includegraphics[width=4.5cm,height=4.5cm]{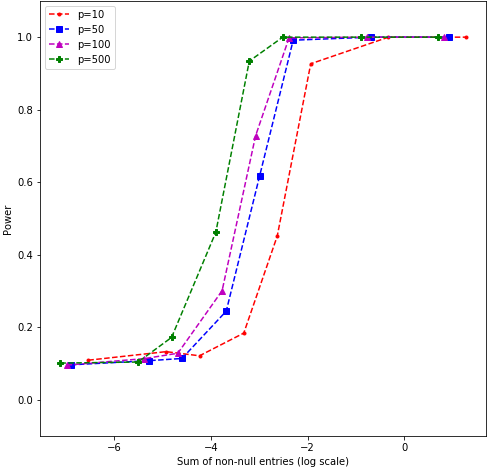}}%
\hspace{\stretch{1}}
\end{minipage}
\begin{minipage}{.32\textwidth}
\centering
\hspace*{\stretch{1}}%
\subfloat[$n=500$]{\includegraphics[width=4.5cm,height=4.5cm]{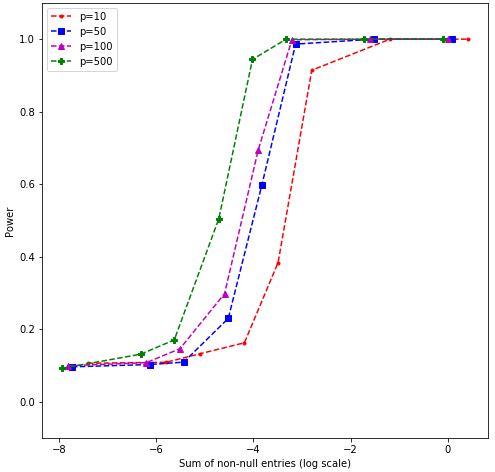}}%
\hspace{\stretch{1}}
\end{minipage}
\begin{minipage}{.32\textwidth}
\centering
\hspace*{\stretch{1}}%
\subfloat[$n=1000$]{\includegraphics[width=4.5cm,height=4.5cm]{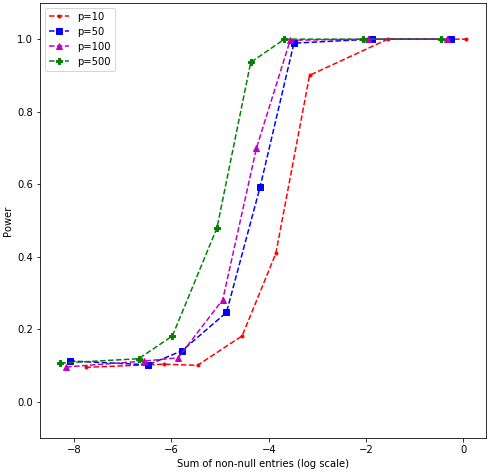}}%
\hspace{\stretch{1}}
\end{minipage}
\caption{Power of the $\Delta^{MS}_n$ test}
\label{fig:MS}
\end{figure}

Figure ~\ref{fig:MS} shows that the $\Delta^{MS}_n$ test shows a similar behaviour as the $\Delta^{MS+}_n$ test. However, the two-sided test efficiency benefits more from the high-dimension $p$ than the one-sided version, in the sense that the curves shift more to the left, towards the small values of the measure of separation when $p$ is large.

Let us consider the $\Delta^{HS+}_n$ test.

\begin{figure}[H]
\centering
\begin{minipage}{.32\textwidth}
\centering
\hspace*{\stretch{1}}%
\subfloat[$n=100$]{\includegraphics[width=4.5cm,height=4.5cm]{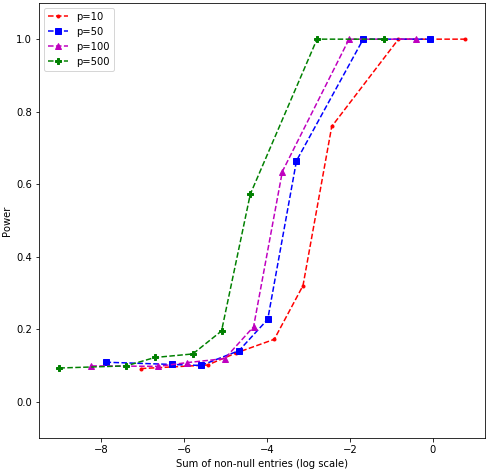}}%
\hspace{\stretch{1}}
\end{minipage}
\begin{minipage}{.32\textwidth}
\centering
\hspace*{\stretch{1}}%
\subfloat[$n=500$]{\includegraphics[width=4.5cm,height=4.5cm]{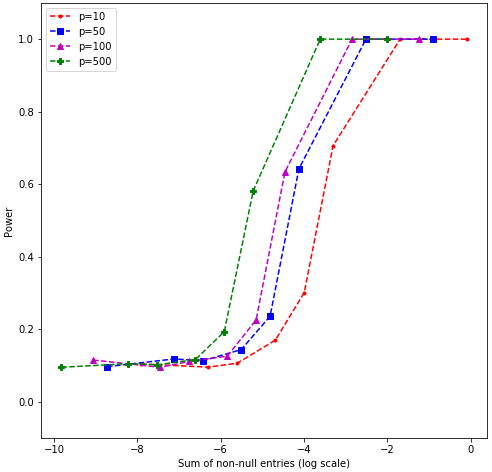}}%
\hspace{\stretch{1}}
\end{minipage}
\begin{minipage}{.32\textwidth}
\centering
\hspace*{\stretch{1}}%
\subfloat[$n=1000$]{\includegraphics[width=4.5cm,height=4.5cm]{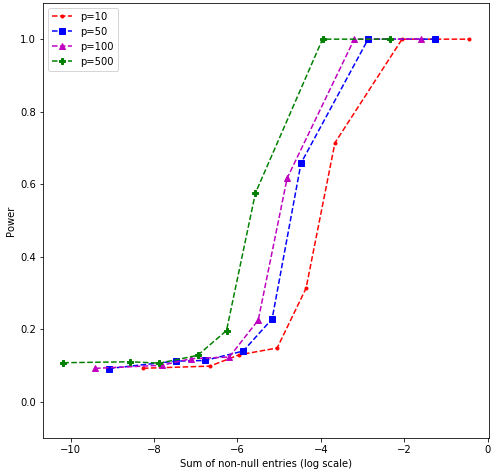}}%
\hspace{\stretch{1}}
\end{minipage}
\caption{Power of the $\Delta^{HS+}_n$ test}
\label{fig:HS+}
\end{figure}

Figure ~\ref{fig:HS+} shows that the $\Delta^{HS+}_n$ test behaves similarly to the $\Delta^{MS+}_n$ and $\Delta^{MS}_n$ tests.

Finally, we consider the two-sided $HS$ test.

\begin{figure}[H]
\centering
\begin{minipage}{.32\textwidth}
\centering
\hspace*{\stretch{1}}%
\subfloat[$n=100$]{\includegraphics[width=4.5cm,height=4.5cm]{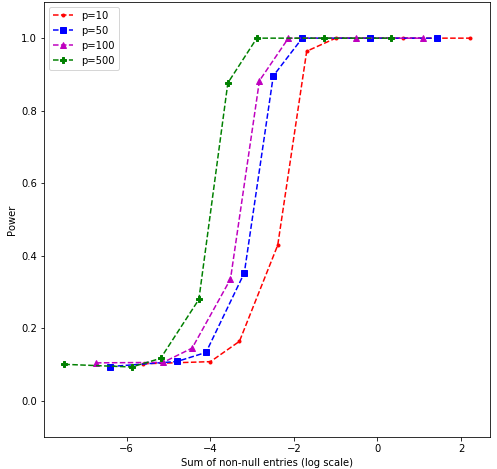}}%
\hspace{\stretch{1}}
\end{minipage}
\begin{minipage}{.32\textwidth}
\centering
\hspace*{\stretch{1}}%
\subfloat[$n=500$]{\includegraphics[width=4.5cm,height=4.5cm]{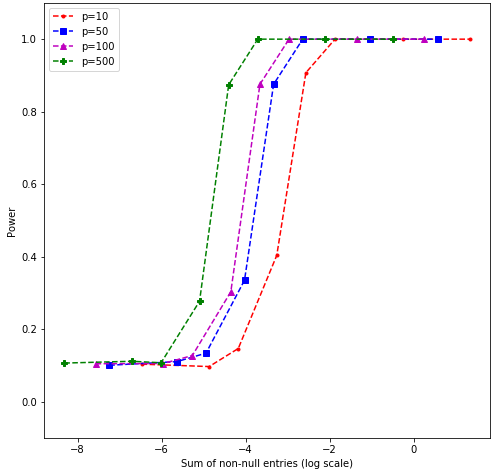}}%
\hspace{\stretch{1}}
\end{minipage}
\begin{minipage}{.32\textwidth}
\centering
\hspace*{\stretch{1}}%
\subfloat[$n=1000$]{\includegraphics[width=4.5cm,height=4.5cm]{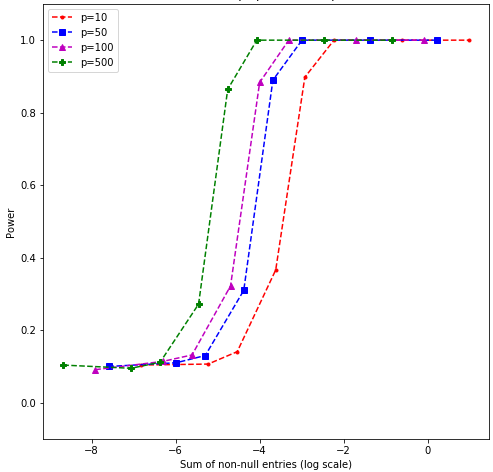}}%
\hspace{\stretch{1}}
\end{minipage}
\caption{Power of the $\Delta^{HS}_n$ test}
\label{fig:HS}
\end{figure}

Figure~\ref{fig:HS} shows that the $\Delta^{HS}_n$ tests also behaves as the previous ones. The high dimension improves the efficiency of the tests. We can also notice that the power of the tests increase rapidly around -3 on the logarithmic scale of the measure of separation.

\subsection{Effect of non null entries}

In the previous Section, we have plotted numerical simulations of the four tests presented in the paper. However we want to understand in more details the impact of the different choices that can be made in this procedures namely: the impact of the number of non null entries $s$, the impact of the location of non-null entries (close to the main diagonal or far from it).

In this sub-section we focus our study on the $\Delta_n^{MS+}$ test as we can extrapolate its behaviour to the other three tests. The underlying covariance matrix belongs to the class $\mathcal{F}_+(s, S, \sigma)$, for some $s \in \llbracket 1; S \rrbracket$.

First, we study the impact of the number of non null entries. For all the previous graphs $s$ was fixed and set to $(S-1)/{2}$. The objective is to observe how the value of $s$ impacts the behaviour of the test. For this purpose we plot side by side the $\Delta^{MS+}_n$ test with $s = S-1$ and $s = (S-1)/{2}$ for $n=100$ and different values of $p$ ($10, 20$ and $50$).

\begin{figure}[H]
\centering
\begin{minipage}{.49\textwidth}
\centering
\hspace*{\stretch{1}}%
\subfloat[$s=S-1$]{\includegraphics[width=4.5cm,height=4.5cm]{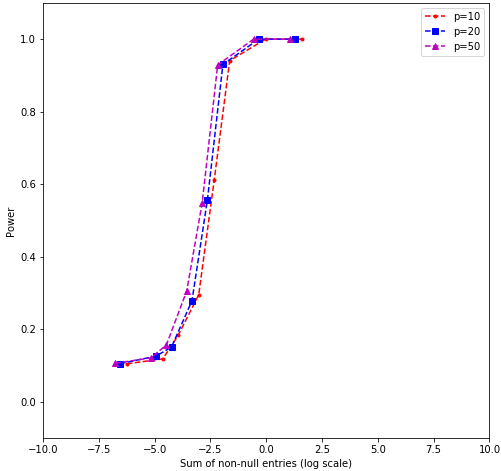}}%
\hspace{\stretch{1}}
\end{minipage}
\begin{minipage}{.49\textwidth}
\centering
\hspace*{\stretch{1}}%
\subfloat[$s=\dfrac{S-1}{2}$]{\includegraphics[width=4.5cm,height=4.5cm]{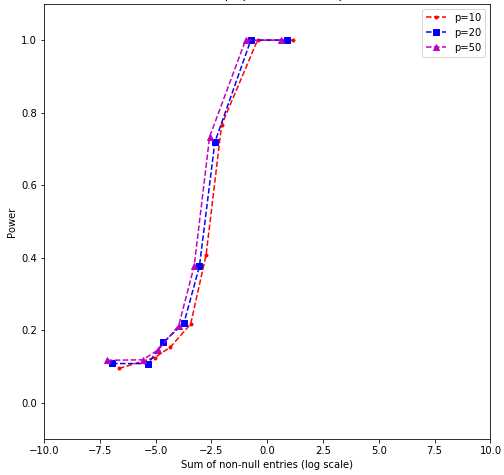}}%
\hspace{\stretch{1}}
\end{minipage}
\caption{Impact of the number of non null entries on $\Delta_n^{MS+}$}
\label{fig:MS_sparse}
\end{figure}

Figure ~\ref{fig:MS_sparse} shows that the number of non null entries has no major impact on the power of the test procedure $\Delta^{MS+}_n$.

\bigskip

Second, we look at the impact of the randomness in the location of the non null entries. In all previous graphs the non null entries were randomly located. The objective is to observe how the location of the non null entries impacts the behaviour of the test. To this end we plot the power function of $\Delta_n^{MS+}$ test with $s = ({S-1})/{2}$ for $n=100$ and different values of $p$. The non null entries are: (a) randomly located, (b) located next to the main diagonal. The plot (c) shows simultaneously the power functions of  $\Delta_n^{MS+}$ test for $p=10$ and $n=100$, but with non null entries randomly chosen i.e $\textbf{S} \subset \llbracket1;S \rrbracket$ with $|\textbf{S}| = s $ (red), fixed next to the main diagonal i.e $\textbf{S} = \llbracket1, s \rrbracket $ (blue) and fixed on the last values of the support i.e $\textbf{S} = \llbracket S-s; S \rrbracket$ (magenta).

\begin{figure}[H]
\centering
\begin{minipage}{.32\textwidth}
\centering
\hspace*{\stretch{1}}%
\subfloat[Randomly chosen]{\includegraphics[width=4.5cm,height=4.5cm]{Pictures/MS+_s1_rand1_log_1.png}}%
\hspace{\stretch{1}}
\end{minipage}
\begin{minipage}{.32\textwidth}
\centering
\hspace*{\stretch{1}}%
\subfloat[Next to the main diagonal]{\includegraphics[width=4.5cm,height=4.5cm]{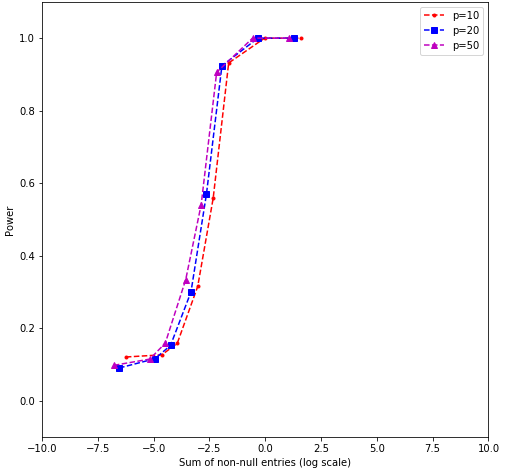}}%
\hspace{\stretch{1}}
\end{minipage}
\begin{minipage}{.32\textwidth}
\centering
\hspace*{\stretch{1}}%
\subfloat[On the same graph]{\includegraphics[width=4.5cm,height=4.5cm]{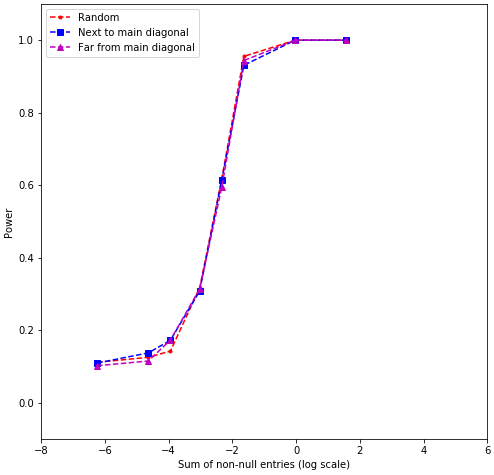}}%
\hspace{\stretch{1}}
\end{minipage}
\caption{Impact of the position of the non null entries on $\Delta_n^{MS+}$}
\label{fig:MS+_random}
\end{figure}

Figure ~\ref{fig:MS+_random} shows that the location of the non null entries has no impact on the $\Delta_n^{MS+}$ test performances.
In conclusion, the tests are sensitive neither to the number of non null entries nor to their location.

\subsection{Comparison between \texorpdfstring{$\Delta_n^{MS}$}{2} and \texorpdfstring{$\Delta_n^{HS}$}{2}}

The four test procedures $\Delta_n^{MS+}$, $\Delta_n^{MS}$, $\Delta_n^{HS+}$ and $\Delta_n^{HS}$ present very similar behaviour of their power curves. However, for high sparsity levels of the covariance matrix $\Delta_n^{HS+}$ and $\Delta_n^{HS}$ were designed to be more efficient than respectively $\Delta_n^{MS+}$ and $\Delta_n^{MS}$. The objective is to observe the difference in their behaviours under such high sparsity levels assumption. In this sub-section we illustrate our study on the two-sided $\Delta_n^{MS}$ and $\Delta_n^{HS}$ tests only, as they are analogous to their one-sided versions.

In order to observe the difference in the impact of sparsity on these two tests we plot their power curves by the number of non null entries $s$. The parameters are set as follows $n=100$, $p=100$ and $S=\sqrt{p}=10$. The plot is repeated for the non null entries common value to be $\sigma = {t_{n, p, \alpha}}/{100} \approx ~ 0.01473$ and $\sigma = {t_{n, p, \alpha}}/{50} \approx 0.02945$.
As the $\Delta_n^{HS}$ test requires a value for $s$ the true value is given in Figure~\ref{fig:MS_HS_known}.

\begin{figure}[H]
\centering
\begin{minipage}{.49\textwidth}
\centering
\hspace*{\stretch{1}}%
\subfloat[$\sigma = \frac{t_{n,p,\alpha}}{100}$]{\includegraphics[width=4.5cm,height=4.5cm]{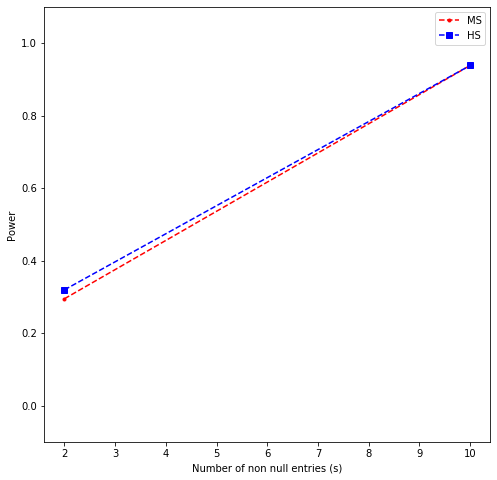}}%
\hspace{\stretch{1}}
\end{minipage}
\begin{minipage}{.49\textwidth}
\centering
\hspace*{\stretch{1}}%
\subfloat[$\sigma = \frac{t_{n,p,\alpha}}{50}$]{\includegraphics[width=4.5cm,height=4.5cm]{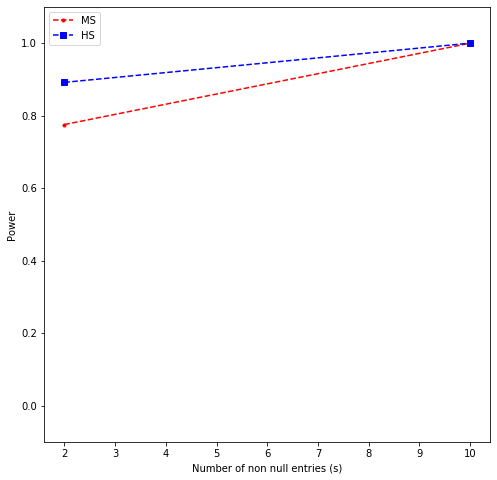}}%
\hspace{\stretch{1}}
\end{minipage}
\caption{$\Delta_n^{MS}$ vs $\Delta_n^{HS}$ with $s$ known}
\label{fig:MS_HS_known}
\end{figure}

Figure ~\ref{fig:MS_HS_known} shows that indeed the $\Delta_n^{HS}$ test procedure with known sparsity $s$ has better detection power than $\Delta_n^{MS}$ for higher sparsity, as it was expected. It can also be noticed that larger significant values of the non-null correlations improve even more the power $\Delta_n^{HS}$ over $\Delta_n^{MS}$.

We build now a new $\Delta_n^{HS}$ procedure that is free of knowledge of $s$ by aggregating several procedures  $\Delta_n^{HS}(s)$ for different values of $s$. Then we compare it to $\Delta_n^{MS}$. 
Consider a grid of plausible values of $s$ from 1 to $S$, build all $\Delta_n^{HS}(s) $ and decide according to
$$
\Delta_n^{HS} = \max_s \Delta_n^{HS},
$$
that is reject whenever at least one of the tests rejected and accept otherwise.

Let us confront the aggregated high-sparsity test and the moderate-sparsity test procedures. The two test procedures have been run in the same setup $n=100$, $p=100$ and $S = \sqrt{p} = 10$. The true values of $s$ are being set to $s=4$ and $s=7$, respectively. We plot the power curves of the two procedures by the measure of separation on a log-scale. The latter is rising because of growing values of $\sigma$.

In both cases, the grid of plausible sparsity levels has been fixed to two values: 2 and 10, which means that
\[
\Delta_n^{HS} = \max\{\Delta_n^{HS}(2),\Delta_n^{HS}(10)\}
\]
even though the true underlying sparsity value is not on the grid. This does not seem to be a drawback.

\begin{figure}[H]
\centering
\begin{minipage}{.49\textwidth}
\centering
\hspace*{\stretch{1}}%
\subfloat[$s=4$]{\includegraphics[width=4.5cm,height=4.5cm]{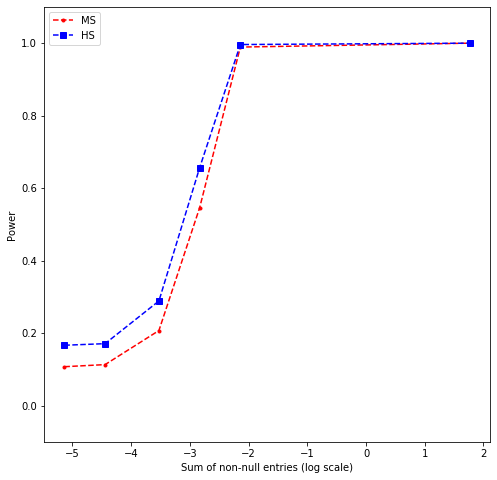}}%
\hspace{\stretch{1}}
\end{minipage}
\begin{minipage}{.49\textwidth}
\centering
\hspace*{\stretch{1}}%
\subfloat[$s=7$]{\includegraphics[width=4.5cm,height=4.5cm]{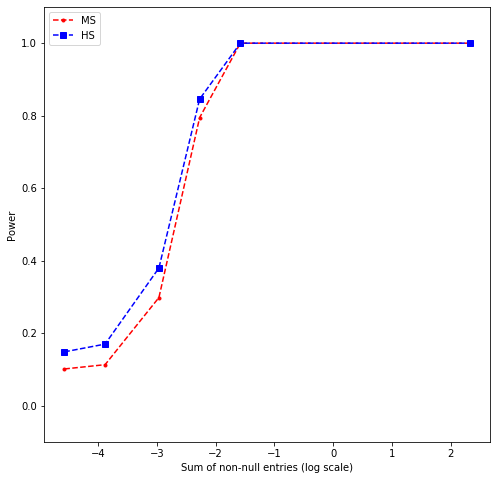}}%
\hspace{\stretch{1}}
\end{minipage}
\caption{$\Delta_n^{MS}$ vs $\Delta_n^{HS}$ with $s$ unknown}
\label{fig:MS_HS_unknown}
\end{figure}

In Figure ~\ref{fig:MS_HS_unknown} it appears that even with unknown value of  $s$ the $\Delta_n^{HS}$ test procedure performs better than $\Delta_n^{MS}$. It can be noticed that the curves show larger differences for lower values of the measure of separation.

In conclusion, the theoretical improvements of highly-sparse over moderately sparse procedures show up in the very extreme cases where the underlying signal is very close to white noise either because of very weak correlations or of very few non-null values.

\subsection{A high-dimensional \texorpdfstring{$MA$ series}{2}}
Let us construct a stationary process belonging to our set of sparse covariance matrices. Consider the stationary process $X_t$ defined by the following moving average ($MA$) model :
\[
X_t=\sum_{i=0}^{\lfloor\frac{p}{4}\rfloor} \phi^i \epsilon_{t-2i}
\]
with $\{\epsilon_t\}_{t \in \mathbb{N}}$ a Gaussian white noise and $|\phi|< 1$. The auto-covariance function of this series is
 \begin{align*}
 \text{Cov}(X_{t+h},X_t) =
 \left\{
 \begin{array}{ll}
 0, & \text{ if } h \text{ odd, or } h \geq \frac p4, \\
 \phi^{-\frac{h}{2}}\left(\frac{\phi^h-\phi^{2\left(\lfloor\frac{p}{4}+1\rfloor\right)}}{1-\phi^2}\right), & \text{ otherwise.}
 \end{array}
 \right.
 \end{align*}

In this example, the $p$-dimensional Gaussian vector $X=\left({X_t},...,{X_{t+p}} \right)$ has a covariance matrix belonging to the class $\mathcal{F}(s,S,\sigma)$ with $s\geq \frac{p}{4}-1$ tending to infinity with $p$, $S\leq \frac{p}{2}$ and \[
\sigma=\phi^{-\frac{1}{2}\lfloor\frac{p}{4}\rfloor}\left(\frac{\phi^{\lfloor\frac{p}{4}\rfloor}-\phi^{2\left(\lfloor\frac{p}{4}+1\rfloor\right)}}{1-\phi^2}\right).
\]

We plot the power of the $\Delta_n^{MS}$ test on the $y$-axis and the value of $\phi < 1$ on the $x$-axis.

\begin{figure}[H]
\centering
\begin{minipage}{.49\textwidth}
\centering
\hspace*{\stretch{1}}%
\subfloat[$n=500$]{\includegraphics[width=4.5cm,height=4.5cm]{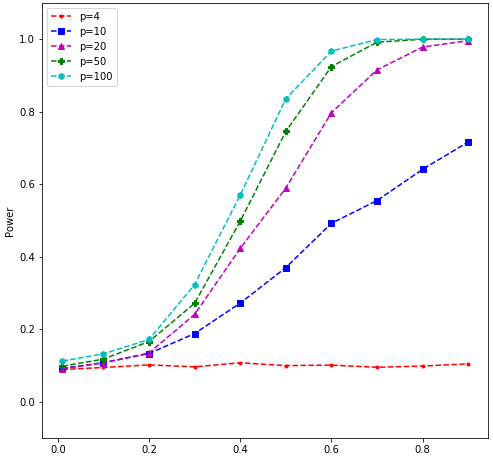}}%
\hspace{\stretch{1}}
\end{minipage}
\begin{minipage}{.49\textwidth}
\centering
\hspace*{\stretch{1}}%
\subfloat[$n=50$]{\includegraphics[width=4.5cm,height=4.5cm]{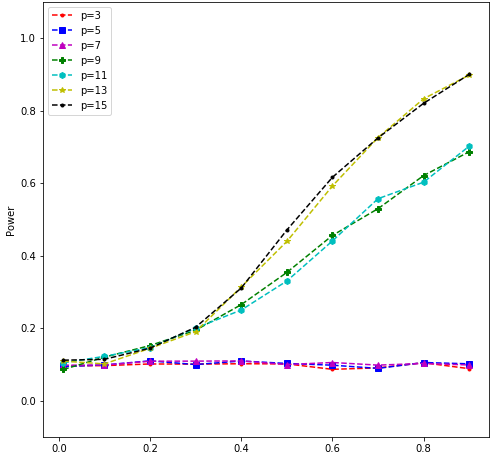}}%
\hspace{\stretch{1}}
\end{minipage}
\caption{Power of $\Delta_n^{MS}$ test for the $MA(\lfloor p/4 \rfloor)$}
\label{fig:example}
\end{figure}
Figure ~\ref{fig:example} shows the power of the $\Delta_n^{MS}$ test for this example for various values of $p$. It can be seen that the $\Delta_n^{MS}$ test performs better when the value of $p$ increases.
We point out that for $p<8$ the $MA(\lfloor p/4 \rfloor)$ is a white noise. It explains why the power of the $\Delta_n^{MS}$ test stays constantly low when $p<8$ .

\section{Proofs}

\subsection{Proof of Theorem \ref{Theorem_subexp}}
	The following lemma is useful to prove the theorem. We prove a more general statement involving an arbitrary constant $K$ in (0,1). It is sufficient to take $K=1/2$ to deduce the theorem.
\begin{lemma}\label{lemma_det}
	Let $\Sigma \in \mathcal{S}_p^{++}$ and $\Sigma^{1/2}$ be its square root. Let $A\in S_p$ and $M=\Sigma^{1/2}A\Sigma^{1/2}$. Then,
	for an arbitrary $K\in]0,1[$,  the matrix ${I_p}-tM$ is invertible and $$\det\left(\left({I_p}-tM\right)\right)^{-1}\leq\exp\left(t\mathrm{Tr}(A\Sigma)+\frac{t^2||A\Sigma||^{2}_F}{2(1-K)}\right), \mbox{ for all } |t|<\frac{K}{||A\Sigma||_{\infty}}.
$$
\end{lemma}
\begin{proof}
	Let $\lambda_1,...,\lambda_p$ be the real eigenvalues of the symmetric matrix $M$ associated to the eigenvectors $x_1,...,x_p$. Then
	for an arbitrary $K\in]0,1[$, for all $|t|<\frac{K}{||A\Sigma||_{\infty}}$,  $1-t\lambda_1,...,1-t\lambda_p$ are the strictly positive eigenvalues of the matrix ${I_p}-tM$ associated to the eigenvectors $x_1,..,x_p$ We have
	\begin{align*}
	\det\left({I_p}-tM\right)^{-1}&=\exp\left(-\sum_{k=1}^{p}\log(1-t\lambda_k)\right)=
	\exp\left(\sum_{k=1}^{p}\sum_{i=1}^{\infty}\frac{1}{i}\left(t\lambda_k\right)^i\right)\\
	 &=\exp\left(t\mathrm{Tr}(A\Sigma)+\sum_{k=1}^{p}t^2\lambda_k^2\left(\sum_{i=0}^{\infty}\frac{t^i}{i+2}\lambda_k^i\right)\right)\\
	 \end{align*}
	 
	 \begin{align*}
	 \det\left({I_p}-tM\right)^{-1} 
	 &\leq\exp\left(t\mathrm{Tr}(A\Sigma)+\sum_{k=1}^{p}\frac{t^2\lambda_k^2}{2}\left(\sum_{i=0}^{\infty}t^i\lambda_k^i\right)\right)\\
	 &=\exp\left(t\mathrm{Tr}(A\Sigma)+\frac{t^2}{2}\sum_{k=1}^{p}\frac{\lambda_k^2}{1-t\lambda_k}\right) . \\
	\end{align*}
	By using the fact that $||A\Sigma||^{2}_F=||M||^{2}_F=\sum_{k=1}^{p}\lambda_k^2$ and that $||A\Sigma||_{\infty}=||M||_{\infty}=\max_k|\lambda_k|$, we have :
	$$	\det\left({I_p}-tM\right)^{-1}\leq \exp\left(t\mathrm{Tr}(A\Sigma)+\frac{t^2||A\Sigma||^{2}_F}{2(1-K)}\right)$$ which ends the proof.
\end{proof}

Let us note that if $X\sim\mathcal{N}(0_p,\Sigma)$,  then $Y=\Sigma^{-1/2}X\sim \mathcal{N}(0_p,{I_p})$.\\
For all $|t|<\frac{nK}{2||A\Sigma||_{\infty}}$, we have :
\begin{align*}	
\mathbb{E}\left[\exp\left(t\varphi_A\left(\Sigma_n-\Sigma\right)\right)\right]
&=\mathbb{E}\left[\exp\left(\frac{t}{n}\left(X^TAX\right) \right)\right]^n\exp(-t\mathrm{Tr}(A\Sigma))\\
&=\mathbb{E}\left[\exp\left(\frac{t}{n}\left(Y^T\Sigma^{T 1/2}A\Sigma^{1/2}Y\right) \right)\right]^n\exp(-t\mathrm{Tr}(A\Sigma))\\
&=\mathbb{E}\left[\exp\left(\frac{t}{n}\left(Y^TMY\right) \right)\right]^n\exp(-t\mathrm{Tr}(A\Sigma))=: T, \mbox{ say}.
\end{align*}
Now, we use the probability density of $Y$ and calculate explicitly
\begin{align*}
T& :=\exp(-t\mathrm{Tr}(A\Sigma))\left(\left(\frac{1}{2\pi}\right)^{p/2}\int ...\int \exp\left(\frac{t}{n}Y^TMY -\frac{1}{2}Y^TY\right)dy_1...dy_p\right)^n\\
&=\exp(-t\mathrm{Tr}(A\Sigma))\left(\left(\frac{1}{2\pi}\right)^{p/2}\int ...\int \exp\left( -\frac{1}{2}Y^T({I_p}-\frac{2t}{n})M)Y\right)dy_1...dy_p\right)^n\\
&=\exp(-t\mathrm{Tr}(A\Sigma))\left(\det\left({I_p}-t \frac{2}{n}M\right)\right)^{-n/2}\\
\end{align*}
By applying Lemma \ref{lemma_det}, we have
\begin{align*}
\mathbb{E}\left[\exp\left(t\varphi_A\left(\Sigma_n-\Sigma\right)\right)\right]
&\leq\exp(-t\mathrm{Tr}(A\Sigma))\exp\left(t\mathrm{Tr}(A\Sigma)+\frac{t^2||A\Sigma||^{2}_F}{n(1-K)}\right)\\&
=\exp\left(\frac{t^2||A\Sigma||^{2}_F}{n(1-K)}\right)=\exp\left(\frac{\nu^2t^2}{2}\right) \text{ with }\nu^2=\frac{2||A\Sigma||_F^2}{n(1-K)}.
\end{align*}

\subsection{Proof of Proposition \ref{prop_norm}}

1. To bound the operator norm of the matrix $A_W$, we use Gershgorin's circle theorem.  Let $M=(m_{i,j})_{1\leq i,j\leq p}$ be a $p\times p$ matrix.  Then, all eigenvalues of the matrix $M$ lie within at least one of the Gershgorin discs $D(m_{ii},\sum_{j\neq{i}} \left|m_{ij}\right|)$.

Gershgorin's circle theorem applied to the matrix $A_W$ gives us :
	\begin{align*}
	||A_W||_\infty=\max_k|\lambda_k| \in D\left(0,2\sum_{j\in W}\frac{1}{2(p-j)}\right) \Rightarrow ||A_W||_\infty\leq \frac{w}{p-S}
	\end{align*}
	To bound the squared Frobenius norm, we sum all the squared elements of $A_W$, which gives us :
	\begin{align*}
	||A_W||_F^2=2\sum_{j\in W}\frac{p-j}{4(p-j)^2}=\sum_{j\in W}\frac{1}{2(p-j)}\leq \frac{w}{2(p-S)}
	\end{align*}

2. To bound the operator norm of the matrix $A_W\Sigma$  for some $\Sigma $ in $\mathcal{F}(s,S, \sigma)$, we use Cauchy-Schwarz inequality together with Gershgorin's circle theorem :$$||A_W \Sigma||_\infty\leq||A_W||_\infty ||\Sigma||_\infty\leq \sigma_0 \frac{(2s+1)w}{p-S}$$ \\
	To bound the squared Frobenius norm of the matrix $A_W\Sigma$ we will use the following lemma.
	\vspace{0.8cm}
	\begin{lemma}\label{lemma_norm}
	Let $M$and $N$ be two $p\times p$ symmetric matrices. Then $ ||MN||_F^2=\mathrm{Tr}(M^2 N^2) $ and
$$||MN||_F^2\leq \max_{1\leq k \leq p}|\lambda_k|^2||N||^2_F=||M||_\infty^2||N||^2_F$$
	\end{lemma}
	\vspace{0.8cm}
\begin{proof}
We have $
||MN||_F^2=\mathrm{Tr}( MNN^TM^T)=\mathrm{Tr}( M^2N^2)$, with $M^2$ and $N^2$ symmetric and positive semi-definite matrices $(M^2\geq 0, \quad N^2\geq 0)$. \\
Recall that, if $A\leq B$ (in the sense that $B-A\geq 0$), then
$\mathrm{Tr}(AC)\leq\mathrm{Tr}(BC)$, for any $C\geq 0$.\\
Here, $M^2\leq \lambda_{max}(M^2){I_p}\leq \lambda_{max}^2(M){I_p}$ and this gives $$\mathrm{Tr}(M^2N^2)\leq \lambda_{max}^2(M)\mathrm{Tr}(N^2)$$
\end{proof}
If $w>1$, using Lemma \ref{lemma_norm} on $M=\Sigma$ and $N=A_W$, we have $$||A_W\Sigma||_F^2\leq||A_W||_F^2||\Sigma||_\infty^2\leq \sigma_0^2 \frac{w(2s+1)^2}{2(p-S)}$$
If $w=1$ and $W=\{j\}$, using Lemma \ref{lemma_norm} on $M^2=\Sigma$ and $N^2=\Sigma^{1/2}A_j^2\Sigma^{1/2}$, we have
$$
||A_j\Sigma||_F^2 = \mathrm{Tr}(A_j^2 \Sigma^2)\leq ||A_j\Sigma^{1/2}||_F^2||\Sigma^{1/2}||_\infty^2\leq \sigma_0 (2s+1)||A_j\Sigma^{1/2}||_F^2
$$
  It suffices to prove that $||A_j\Sigma^{1/2}||_F^2=\mathrm{Tr}(A_j^2\Sigma)\leq \sigma_0 \frac{\mathcal{K}}{(p-S)}$ so that we conclude the proof that $||A_j\Sigma||_F^2\leq \sigma_0^2  \frac{\mathcal{K}(2s+1)}{p-S}$.
 Let $B_j=A_j^2=(b_{k,l}^j)_{1\leq k,l\leq p}$. For every $1\leq k,l\leq p$, we have
 \begin{equation*}
 b_{k,l}^j=\sum_{i=1}^{p}a_{k,i}^ja_{i,l}^j=\sum_{i=1}^{p}a_{|k-i|}^ja_{|l-i|}^j=\sum_{i=1}^{p}\frac{\delta_{|k-i|=j}\delta_{|l-i|=j}}{4(p-j)^2}
 \end{equation*}
\begin{itemize}
\item If $k=l$, $b_{k,k}^j=\left\{\begin{array}{c}
\frac{1}{2(p-j)^2} \text{ if } j < \frac{p}{2} \quad \text{and } j<k\leq p-j \\
0\quad  \text{ if } j \geq \frac{p}{2} \quad \text{and}\quad p-j\leq k<j \\
\frac{1}{4(p-j)^2} \quad\text{ otherwise}\hspace{3cm}
\end{array}\right.$

 \item If $k\neq l$, for $\delta_{|k-i|=j}\delta_{|l-i|=j}$ to be non-null, we need :$$\left\{\begin{array}{c}
k-i=j\text{ and }l-i=-j \\
\text{or}\\
l-i=j\text{ and }k-i=-j\\
\end{array}\right.\Leftrightarrow \left\{\begin{array}{c}
k-l=2j\text{ and }i=\frac{k+l}{2} \\
\text{or}\\
l-k=2j\text{ and }i=\frac{k+l}{2}\end{array}\right.\\\Leftrightarrow \left\{\begin{array}{c}
|k-l|=2j \\
\text{and}\\
i=\frac{k+l}{2}\\
\end{array}\right.$$
Therefore, $
b_{k,l}^j=\left\{\begin{array}{c}
\frac{1}{4(p-j)^2} \text{ if } j<\frac{p}{2} \text{ and } |k-l|=2j \\
0 \text{ otherwise }\hspace{2.4cm}
\end{array}\right.$
\end{itemize}
\vspace{0.8cm}
Summing up the results gives us
\begin{align*}
||A_j\Sigma^{1/2}||_F^2&=\mathrm{Tr}(A_j^2\Sigma)=\sum_{m=1}^{p}\left(\sum_{i=1}^{p}b_{m,i}\sigma_{i,m}\right)\\
&\leq \sigma_0 \sum_{m=1}^{p}\left(\sum_{i=1}^{p}b_{m,i}\right)
=\sigma_0 \sum_{m=1}^{p}b_{m,m} + \sigma_0 \sum_{m\neq i} b_{m,i}\\
&\leq \sigma_0  \left\{\begin{array}{c}
\frac{2(p-j)+2(p-2j)}{4(p-j)^2} \quad \text{if}\quad j<\frac{p}{2}\\
\frac{2(p-j)}{4(p-j)^2} \quad \text{otherwise}
\end{array}\right.\\
&\leq \sigma_0  \left\{\begin{array}{c}
\frac{1}{(p-j)} \quad \text{if}\quad j<\frac{p}{2}\\
\frac{1}{2(p-j)} \quad \text{otherwise}
\end{array}\right.
\end{align*}
 This means that $$||A_j\Sigma||_F^2\leq \sigma_0 (2s+1)||A_j\Sigma^{1/2}||_F^2\leq \sigma_0^2 \frac{\mathcal{K}(2s+1)}{(p-S)}$$ \\
where $\mathcal{K}=\left\{\begin{array}{c}
1,\quad \text{if} \quad W \subseteq\{1,...,\frac{p}{2}-1\}\\
\frac{p}{2},\quad \text{if} \quad W \subseteq\{\frac{p}{2},...,p\}\hspace{0.5cm}	
\end{array}\right.$

	\subsection{Proof of Theorem \ref{Theorem_testMS}}
	We know from Corollary \ref{corollary_1} that the type $I$ error probability is such that
			\begin{equation*}
	\mathbb{P}_{I_p}\left[\varphi_{A_{1:S}}(\Sigma_n-{I_p})\geq t_{n,p}^{MS+}\right]\leq \exp\left(-\frac{u}{4}\right)
	\end{equation*} and that, for any $\Sigma$ in $\mathcal{F}_+(s, S, \sigma)$, we have
\begin{equation*}
	\mathbb{P}_{\mathrm{\Sigma}}[\varphi_{A_{1:S}}\left(\Sigma_n-\Sigma\right)\geq (1+2s)t_{n,p}^{MS+}]\leq \exp\left(-\frac{u}{4}\right),\quad\text{for all}\quad u>0.
\end{equation*}
We can bound the type $II$ error probability under the assumption that
$\sigma \geq\frac{2(s+1)}{s}t_{n,p}^{MS+}$ :
	\begin{align*}
	&\mathbb{P}_\mathrm{\Sigma}\left[\varphi_{A_{1:S}}(\Sigma_n-{I_p})\leq t_{n,p}^{MS+}\right]=\mathbb{P}_\mathrm{\Sigma}\left[\varphi_{A_{1:S}}(\Sigma_n-\mathrm{\Sigma})\leq t_{n,p}^{MS+}-\varphi_{A_{1:S}}(\mathrm{\Sigma})\right]\\
	&=\mathbb{P}_\mathrm{\Sigma}\left[\varphi_{A_{1:S}}(\mathrm{\Sigma}-\Sigma_n)\geq \varphi_{A_{1:S}}(\mathrm{\Sigma})-t_{n,p}^{MS+}\right]\\
	&\leq\mathbb{P}_\mathrm{\Sigma}\left[\varphi_{A_{1:S}}(\mathrm{\Sigma}-\Sigma_n)\geq s\sigma-t_{n,p}^{MS+}\right]\\
	&\leq\mathbb{P}_\mathrm{\Sigma}\left[\varphi_{A_{1:S}}(\mathrm{\Sigma}-\Sigma_n)\geq (2s+1)t_{n,p}^{MS+}\right]	 \leq \exp\left(-\frac{u}{4}\right),\quad\text{for all}\quad u>0.	
	\end{align*}
	Finally :$$R(\Delta_n^{MS+},\mathcal{F}^+)=\mathbb{P}_{I_p}(\varphi_{A_{1:S}}(\Sigma_n-{I_p})\geq t_{n,p}^{MS+})+\sup_{\Sigma \in \mathcal{F}^+}\mathbb{P}_\Sigma(\varphi_{A_{1:S}}(\Sigma_n-{I_p})\leq t_{n,p}^{MS+})\leq 2\exp\left(-\frac{u}{4}\right)$$

	\subsection{Proof of Theorem \ref{Theorem_testMS2}}
	Similarly to the proof of Theorem \ref{Theorem_testMS}, we use Corollary \ref{corollary_1} to bound the type $I$ error probability
	\begin{align*}
	&\mathbb{P}_{{I_p}}\left[\sum_{i=1}^{S}|\varphi_{A_i}(\Sigma_n-{I_p})|\geq t_{n,p}^{MS}\right]\leq \mathbb{P}_{I_p}\left[\bigcup_{i=1}^S\left\{|\varphi_{A_i}(\Sigma_n-{I_p})|\geq \frac{t_{n,p}^{MS}}{S}\right\}\right]\\
	&\leq\sum_{i=1}^{S} \mathbb{P}_{I_p}\left[|\varphi_{A_i}(\Sigma_n-{I_p})|\geq \frac{t_{n,p}^{MS}}{S}\right]\\
	&=\sum_{i=1}^{S} \mathbb{P}_{{I_p}}\left[|\varphi_{A_i}(\Sigma_n-{I_p})|\geq \max\left\{\sqrt{\frac{u}{2(1-K)}}\sqrt{\frac{4\log(S)}{n(p-S)}},\frac{u}{K}\frac{4\log(S)}{n(p-S)}\right\}\right]\\
	&\leq \sum_{i=1}^{S} 2\exp\left(-u\log S\right)\\
	&=2\exp\left(-(u-1)\log S\right)
	\end{align*}
	To bound the type $II$ error probability, we use the condition on $\sigma$ :
	\begin{align*}
	&\mathbb{P}_\mathrm{\Sigma}\left[\sum_{i=1}^{S}|\varphi_{A_i}(\Sigma_n-\mathrm{Id})|\leq t_{n,p}^{MS}\right]\leq \mathbb{P}_{\Sigma}\left[\bigcap_{i=1}^S\left\{|\varphi_{A_i}(\Sigma_n-\mathrm{Id})|\leq t_{n,p}^{MS}\right\}\right]\\
	&\leq\sup_{1\leq i \leq S}\mathbb{P}_\Sigma\left[|\varphi_{A_i}(\Sigma_n-{I_p})|\leq t_{n,p}^{MS}\right]\\
	&\leq\sup_{1\leq i \leq S}\mathbb{P}_\Sigma\left[|\varphi_{A_i}(\Sigma_n-\mathrm{\Sigma})|\geq|\varphi_{A_i}(\Sigma-{I_p})|-t_{n,p}^{MS}\right]\\
	&\leq\sup_{1\leq i \leq S}\mathbb{P}_\Sigma\left[|\varphi_{A_i}(\Sigma_n-\mathrm{\Sigma})|\geq\sigma-t_{n,p}^{MS}\right]\\
	&\leq\sup_{1\leq i \leq S}\mathbb{P}_\Sigma\left[|\varphi_{A_i}(\Sigma_n-\mathrm{\Sigma})|\geq\max\left\{\sqrt{\frac{u-1}{2(1-K)}}\sqrt{\frac{4\log S(2s+1)}{n(p-S)}},\frac{(u-1)}{K}\frac{4\log S(2s+1)}{n(p-c)}\right\}\right]\\
	&\leq 2\exp\left(-(u-1)\log S\right)
	\end{align*}
	This gives us finally : $$R(\Delta_n^{MS},\mathcal{F})\leq 4\exp\left(-(u-1)\log S\right)$$

\subsection{Proof of Theorem \ref{theorem_testHS}}
The type $I$ error probability is bounded by
\begin{align*}
\mathbb{P}_{\mathrm{{I_p}}}[\Delta_n^{HS+}=1]&\leq\sum_{\mathbf{S} \subseteq \{1,...,S\},\#\mathbf{S} = s}\mathbb{P}_{I_p}\left[ \varphi_{A_\mathbf{S}}(\Sigma_n-{I_p})\geq t_{n,p}^{HS+} \right]\\
&\leq \sum_{\mathbf{S} \subseteq \{1,...,S\},\#\mathbf{S} = s}\exp\left(-u\log{\binom{S}{s}}\right)\\
&=\exp\left(-(u-1)\log{\binom{S}{s}}\right)
\end{align*}
while the type $II$ error probability is bounded by
\begin{align*}
\mathbb{P}_{\mathrm{\Sigma}}[\Delta_n^{HS+}=0]&=\sup_{\Sigma \in \mathcal{F}^+(s,S,p,\sigma)}\mathbb{P}_\Sigma\left[\bigcap_{\mathbf{S} \subseteq \{1,...,S\},\#\mathbf{S} = s }\{|\varphi_{A_\mathbf{S}}(\Sigma_n-{I_p})|\leq t_{n,p}^{HS+} \}\right]\\
&\leq\sup_{\Sigma \in \mathcal{F}^+(s,S,p,\sigma)}\mathbb{P}_\Sigma\left[\varphi_{A_\mathbf{S}}(\Sigma_n-\mathrm{\Sigma})+\varphi_{A_\mathbf{S}}(\Sigma-{I_p})\leq t_{n,p}^{HS+}\right]\\
&=\sup_{\Sigma \in \mathcal{F}^+(s,S,p,\sigma)}\mathbb{P}_\Sigma\left[\varphi_{A_\mathbf{S}}(\mathrm{\Sigma}-\Sigma_n)\geq \varphi_{A_\mathbf{S}}(\Sigma)-t_{n,p}^{HS+}\right]\\
&\leq\sup_{\Sigma \in \mathcal{F}^+(s,S,p,\sigma)}\mathbb{P}_\Sigma\left[\varphi_{A_\mathbf{S}}(\mathrm{\Sigma}-\Sigma_n)\geq s\sigma-t_{n,p}^{HS+}\right]
\end{align*} for an arbitrary set $\mathbf{S}$ in $\{1,...,S\}$ containing s values. \\
Under the condition $s\sigma-t_{n,p}^{HS+} \geq (2s+1)\max\left\{\sqrt{\frac{u}{2(1-K)}}\sqrt{\frac{s}{n(p-S)}},\frac{u}{K}\frac{s}{n(p-S)}\right\}$ and Corollary \ref{corollary_1}, we have :
$$\mathbb{P}_{\mathrm{\Sigma}}[\Delta_n^{HS+}=0]\leq\sup_{\Sigma \in \mathcal{F}^+(s,S,p,\sigma)}\mathbb{P}_\Sigma\left[\varphi_{A_\mathbf{S}}(\mathrm{\Sigma}-\Sigma_n)\geq \tilde{t}\right]\leq \exp\left(-\frac{u}{4}\right)$$

	\subsection{Proof of Theorem \ref{theorem_testHSG}}
	The proof is similar to the proof of Theorem \ref{Theorem_testMS2}. The type $I$ probability error is bounded by
		\begin{align*}
		\mathbb{P}_{{I_p}}[\Delta_n^{HS}=1]&\leq\sum_{\mathbf{S} \subseteq \{1,...,S\},\#\mathbf{S} = s}\mathbb{P}_{I_p}\left[\sum_{i \in \mathbf{S}} |\varphi_{A_i}(\Sigma_n-{I_p})|\geq t_{n,p}^{HS} \right]\\
		&\leq \sum_{\mathbf{S} \subseteq \{1,...,S\},\#\mathbf{S} = s}\sum_{i \in \mathbf{S}}\mathbb{P}_{I_p}\left[ |\varphi_{A_i}(\Sigma_n-{I_p})|\geq \frac{t_{n,p}^{HS}}{s} \right]\\
		&\leq \sum_{\mathbf{S} \subseteq \{1,...,S\},\#\mathbf{S} = s}\sum_{i \in \mathbf{S}}2\exp\left[-u\log\left(s{\binom{S}{s}}\right)\right]\\
		&=2\exp\left[-(u-1)\log\left(s{\binom{S}{s}}\right)\right]
	\end{align*}
	The type $II$ probability is bounded by \begin{align*}
&\mathbb{P}_{\mathrm{\Sigma}}[\Delta_n^{HS}=0]=\mathbb{P}_\mathrm{\Sigma}\left[\max_{\mathbf{S} \subseteq \{1,...,S\},\#\mathbf{S} = s}\sum_{i\in \mathbf{S}}|\varphi_{A_i}(\Sigma_n-\mathrm{Id})|\leq t_{n,p}^{HS}\right]\\
&\leq \mathbb{P}_{\Sigma}\left[\bigcap_{\mathbf{S} \subseteq \{1,...,S\},\#\mathbf{S} = s}\bigcap_{i \in \mathbf{S}}\left\{|\varphi_{A_i}(\Sigma_n-\mathrm{Id})|\leq t_{n,p}^{HS}\right\}\right]\\
&\leq\sup_{\mathbf{S} \subseteq \{1,...,S\},\#\mathbf{S} = s}\sup_{i \in \mathbf{S}}\mathbb{P}_\Sigma\left[|\varphi_{A_i}(\Sigma_n-\mathrm{\Sigma})|\geq\sigma-t_{n,p}^{HS}\right]\\
&\leq 2\exp\left[-(u-1)\log\left(s{\binom{S}{s}}\right)\right]
\end{align*}

\vspace{-.1cm}
\subsection{Proof of Theorem \ref{theorem_select}}
Using Theorem \ref{Theorem_subexp} and Proposition \ref{prop_norm}, we have :
\begin{align*}
R^{LS}(\hat{\eta},\mathcal{F}^+)=&\sum_{j=1}^{S}\mathbb{E}_{\Sigma}[|\hat{\eta_{j}}-\eta_{j}|]
=\sum_{j \in \mathbf{S} }\mathbb{E}_{\Sigma}[|\hat{\eta_{j}}-\eta_{j}|] + \sum_{j \notin \mathbf{S}, j\leq S }\mathbb{E}_{\Sigma}[|\hat{\eta_{j}}-\eta_{j}|]\\
=&\sum_{j \in \mathbf{S} }\mathbb{E}_{\Sigma}[|\hat{\eta_{j}}-1|] + \sum_{j \notin \mathbf{S}, j\leq S }\mathbb{E}_{\Sigma}[|\hat{\eta_{j}}|]\\
=&\sum_{j \in \mathbf{S} }\mathbb{P}_{\Sigma}[|\varphi_{A_j}\left(\Sigma_n\right)|<\tau_n] + \sum_{j \notin \mathbf{S}, j\leq S }\mathbb{P}_{\Sigma}[|\varphi_{A_j}\left(\Sigma_n\right)|>\tau_n]\\
\leq&\sum_{j \in \mathbf{S} }\mathbb{P}_{\Sigma}[|\varphi_{A_j}\left(\Sigma_n-\Sigma\right)|>\varphi_{A_j}\left(\Sigma\right)-\tau_n] + \sum_{j \notin \mathbf{S}, j\leq S }\mathbb{P}_{\Sigma}[|\varphi_{A_j}\left(\Sigma_n-\Sigma\right)|>\tau_n]\\
\leq&\sum_{j \in \mathbf{S} }\mathbb{P}_{\Sigma}[|\varphi_{A_j}\left(\Sigma_n-\Sigma\right)|>\sigma-\tau_n] + \sum_{j \notin \mathbf{S}, j\leq S }\mathbb{P}_{\Sigma}[|\varphi_{A_j}\left(\Sigma_n-\Sigma\right)|>\tau_n]\\
\leq& \sum_{j \in \mathbf{S} }\mathbb{P}_{\Sigma}\left[|\varphi_{A_j}\left(\Sigma_n-\Sigma\right)|>\max\left\{\sqrt{2u\log(s)}\frac{||A_j\Sigma||_F}{\sqrt{n}},2u\log(s)\frac{||A_j\Sigma||_\infty}{n}\right\}\right] \\&+ \sum_{j \notin \mathbf{S}, j\leq S }\mathbb{P}_{\Sigma}\left[|\varphi_{A_j}\left(\Sigma_n-\Sigma\right)|>\max\left\{\sqrt{2u\log(S-s)}\frac{||A_j\Sigma||_F}{\sqrt{n}}, 2u\log(S-s)\frac{||A_j\Sigma||_\infty}{n}\right\}\right]\\
\leq& \sum_{j \in \mathbf{S} }2\exp\left(-\frac{u\log(s)}{4}\right) + \sum_{j \notin \mathbf{S}, j\leq S }2\exp\left(-\frac{u\log(S-s)}{4}\right)\\
\leq& 2\exp\left(-(u-1)\frac{\log(s)}{4}\right) + 2\exp\left(-(u-1)\frac{\log(S-s)}{4}\right)\\
\end{align*}


\bibliographystyle{plain}
\bibliography{references}

\end{document}